\documentclass[preprint,a4j,leqno]{elsarticle}
\usepackage{amsmath,amssymb}
\usepackage{amsthm}

\usepackage{graphicx,color}

\usepackage[top=30truemm,bottom=30truemm,left=30truemm,right=30truemm]{geometry}

\usepackage{color}
\graphicspath{{./pic/}}

\usepackage[version=3]{mhchem}

\journal{}

\usepackage{bm}

\usepackage{makeidx}
\makeindex

\makeatletter

\newcommand{\Proofname}{Proof}
\makeatother

\newtheorem{theorem}{Theorem}[section]
\newtheorem{lemma}[theorem]{Lemma}
\newtheorem{corollary}[theorem]{Corollary}
\newtheorem{proposition}[theorem]{Proposition}
\newtheorem{Definition}[theorem]{Definition}
\newtheorem{Rem}[theorem]{Remark}

\newenvironment{remark}{\begin{Rem}\rm}{\end{Rem}}

\newcommand\be{\begin{equation}}
\newcommand\ee{\end{equation}}
\newcommand\bea{\begin{eqnarray}}
\newcommand\eea{\end{eqnarray}}
\newcommand\beaa{\begin{eqnarray*}}
\newcommand\eeaa{\end{eqnarray*}}
\newcommand\bay{\begin{array}}
\newcommand\eay{\end{array}}
\newcommand\ba{\begin{align}}
\newcommand\ea{\end{align}}
\newcommand\beba{\begin{equation}\left\{\begin{array}{rcl}}
\newcommand\eeba{\end{array}\right.\end{equation}}
\newcommand\bebaa{\begin{equation*}\left\{\begin{array}{rcl}}
\newcommand\eebaa{\end{array}\right.\end{equation*}}
\newcommand\beca{\begin{equation}\left\{\begin{array}{rcll}}
\newcommand\eeca{\end{array}\right.\end{equation}}
\newcommand\becaa{\begin{equation*}\left\{\begin{array}{rcll}}
\newcommand\eecaa{\end{array}\right.\end{equation*}}

\newcommand{\eps}{\varepsilon}

\newcommand{\R}{\mathbb{R}}
\newcommand{\Z}{\mathbb{Z}}
\newcommand{\N}{\mathbb{N}}

\newcommand{\ds}{\displaystyle}

\graphicspath{{./pic/}}

\begin{document}
\begin{frontmatter}
\title{Keller--Segel type approximation for nonlocal Fokker--Planck equations in one-dimensional bounded domain}

\author[Ryukoku]{Hideki Murakawa}
\ead{murakawa@math.ryukoku.ac.jp}
\author[Hakodate]{Yoshitaro Tanaka\corref{cor}}
\ead{yoshitaro.tanaka@gmail.com}

\address[Ryukoku]{
Faculty of Advanced Science and Technology, 
Ryukoku University,
1-5 Yokotani, Seta Oe-cho, Otsu, Shiga 520-2194, Japan
}

\address[Hakodate]{
Department of Complex and Intelligent Systems,
School of Systems Information Science,
Future University Hakodate,
116-2 Kamedanakano-cho, Hakodate, Hokkaido 041-8655, Japan
}

\cortext[cor]{Corresponding author.}

\begin{abstract}
Numerous evolution equations with nonlocal convolution-type interactions have been proposed.
In some cases, a convolution was imposed as the velocity in the advection term.
Motivated by analyzing these equations, we approximate advective nonlocal interactions as local ones, thereby converting the effect of nonlocality.
In this study, we investigate whether the solution to the nonlocal Fokker--Planck equation can be approximated using the Keller--Segel system.
By singular limit analysis, we show that this approximation is feasible for the Fokker--Planck equation with any potential and that the convergence rate is specified.
Moreover, we provide an explicit formula for determining the coefficient of the Lagrange interpolation polynomial with Chebyshev nodes.
Using this formula, 
the Keller--Segel system parameters for the approximation are explicitly specified by the shape of the potential in the Fokker--Planck equation.
Consequently, we demonstrate the relationship between advective nonlocal interactions and a local dynamical system.
\end{abstract}


\end{frontmatter}

\section{Introduction}
The development of multi-cellular organisms, cell migration, and information processing by nerve cells in the brain depend on various interactions between the cells and other biological factors.
In these phenomena, the functions of these interactions determine the state of the subsequent time evolution and exhibit the occurrence and various behavior of the patterns.
When modeling such phenomena, there are cases in which long-range interactions that affect distant objects globally in space naturally appear.
These interactions are called nonlocal interactions.
They have attracted considerable attention in various fields and have been studied extensively.
The existence of nonlocal interactions has been experimentally suggested, for example, in phenomena such as neural firing in the brain \cite{K1953}, pigmentation patterns in animal skin \cite{NTKK2009,WK2014,YK2014}, development of multicellular organisms, cell migration, and adhesion \cite{KHSIMKANTT2016}.

In the experiment by Kuffler \cite{K1953}, the electrode was set at a ganglion cell in the receptive field of the retina of a cat.
The firing rate against the light stimulus was measured by illuminating two points at different distances from the electrodes.
 Ganglion cells respond to light stimuli locally in space and, conversely, inhibit them laterally in space.
From these observations, by considering the local excitation and lateral inhibition as positive and negative values, respectively, this interaction can be modeled by a sign-change function with radial symmetry.
This function is called the local activation and lateral inhibition (LALI) interaction or Mexican hat.

Experimental results on the interactions between yellow and black pigment cells in the skin of zebrafish were reported by Nakamasu et al. \cite{NTKK2009}.
During these experiments, a square section of black pigmented cells within a zebrafish's skin stripe was eliminated using laser ablation.
The regions of yellow pigmented cells surrounding the squares were removed by laser ablation.
Several zebrafish were prepared with different patterns of yellow pigmented cells, which were removed using a laser.
For each pattern of yellow pigment cells, the number of black pigment cells proliferating in the square located at the center of the pattern was quantified for two weeks.
The comparison of the number of proliferating black pigment cells revealed the existence of both long- and short-range interactions between these cells.
Moreover, the derived interactions are summarized as a network in \cite{NTKK2009}.
A theoretical method to reduce a given reaction-diffusion network with spatial interactions, such as metabolites or signals with arbitrary factors, into the shape of an essential integral kernel have been proposed by Ei et al. \cite{EIKMT2021}.
A Mexican hat function was theoretically derived by applying this reduction method to the network given by \cite{NTKK2009}.
Additionally, cells such as the pigment cells in the zebrafish extend their cellular projections to exchange the biological signals each other.
Here we refer to the papers by Katsunuma et al. \cite{KHSIMKANTT2016} and Kondo \cite{kondo2017}.
Hamada et al. \cite{HWLNHPK2014} and Watanabe and Kondo \cite{WK2014} reported that pigment cells send different biological signals depending on the length of their cellular projections.

Katsunuma et al. \cite{KHSIMKANTT2016} investigated the behavior  of cell adhesion was investigated by using two types of cellular adhesion molecules in HEK 293 cells. 
These cells also have cellular projections that are ten times longer than their body size.
Cells can sense the cell density around them using their total body from the tip of the leading edge, and they can decide the directions of cell migration and cell adhesion.

Based on this biological background, numerous mathematical models have been proposed and analyzed.
The nonlocal interactions are often modeled by convolution with a suitable integral kernel.
Here, we introduce two types of the model of the nonlocal interaction.
We call the convolution itself without a derivative the normal type, and the gradient of convolution in the advection term the advective type.
We introduce mathematical models with a normal type of nonlocal interaction.
The peaks of biological signals are located distantly from the center of the cell body during signal transduction through cellular projections. 
Here, we refer to the observational results of Hamada et al. \cite{HWLNHPK2014} and Watanabe and Kondo \cite{WK2014}.
For these interactions, many models that impose a convolution with an integral kernel with peaks distant from the origin have been proposed.
For example, as models of pigmentation patterns in animal skin \cite{NTY2017,NTY2018}, population dispersal for biological individuals \cite{EGIW2020, HMMG2003}, and vegetation patterns \cite{AIM2018}, the following nonlocal evolution equation is proposed:
\begin{equation*}
	\frac{\partial u}{\partial t} = d\Delta u +K * u +f( u ),
\end{equation*}
where $d \ge 0$ is a constant, $u=u(x,t)$ is the density, $f$ is a suitable reaction or growth term, $K$ is an integral kernel, and $*$ denotes the convolution of two functions in the space variable:
\begin{equation*}
	K*u (x,t) = \int K(x-y) u(y,t) dy.
\end{equation*}
Analytical results were reported by Bates et al. \cite{BFRW1997}, Coville et al. \cite{CDM2008} and Ei et al. \cite{EGIW2020}.
It is rigorously shown by Ninomiya et al. \cite{NTY2017} that the nonlocal interaction plays a role to induce the Turing instability.
By imposing a convolution using the Heaviside function, a nonlocal evolution equation to investigate the dynamics of the membrane potential of neurons in the brain by Amari \cite{A1977}.
Additionally, motivated by the pattern formations observed in animal skins, a nonlocal model by applying the cut function to the convolution term was proposed by Kondo \cite{kondo2017}.
This model can reproduce various patterns by changing only the kernel shape, even though it comprises only one component.
The above nonlocal interactions can be derived from the continuation of spatially discretized models with intercellular interactions by Ei et al. \cite{EISTWY2020}.

Next, we introduce mathematical models of the advective type of nonlocal interaction.
As a first example, the aggregation-diffusion equation was proposed and analyzed for cell migration and collective motion by Bailo et al. \cite{BCMS2020} and Carrillo et al. \cite{CCY2019}:
\begin{equation*}
	\frac{\partial  \rho}{ \partial t}  = \Delta \rho^m - \nabla \cdot (\rho \nabla( W* \rho) ), \quad m\ge 1,
\end{equation*}
where $\rho=\rho(x,t)$ denotes the cell density at  position $x$ at time $t>0$ and $m\ge1$ is a constant.
If the potential $W$ is positive, the convolution term $\nabla( W* \rho)$ determines the velocity of the advection by integrating the gradient of $\rho$.
The density $\rho$ at each point is advected toward the gradient of $\rho$.
Thus, the second term provides the aggregation effect.
If $W$ is positive with a compact support, then the compact support corresponds to the total cell body.
Subsequently, the term $\nabla( W* \rho) $ provides the effect determines the  velocity of the advection by sensing the cell density gradient in the total cell body.
When $m=1$, this model can be classified as a nonlocal Fokker--Planck equation, whereas when  $m>1$, it can be classified as a nonlocal porous medium equation.

Another example is the cell adhesion model.
It has been proposed to describe and analyze the cell adhesion phenomena by Carrillo et al. \cite{CMSTT2019}:
\begin{equation*}
	\frac{\partial  \rho}{ \partial t} = \Delta \rho^2  - \nabla \cdot (\rho(1-\rho)\nabla( W*\rho))+f(\rho).
\end{equation*}
In contrast to the aggregation-diffusion equation, the velocity in the advection term is saturated by the cell density $(1-\rho)$ in this model.
Carrillo et al. \cite{CMSTT2019} reported that this model can replicate cell adhesion and cell sorting phenomena both qualitatively and quantitatively.
For cell migration and cell adhesion processes, the integral kernel is called the potential, and the Mexican hat (or LALI) function is crucial in the local attraction and lateral repulsion in these two models.

Cell pattern formation, migration, and adhesion play pivotal roles in the biological development of various organs and tissues.
Therefore, revealing the mechanisms underlying these phenomena is an important problem.
However, the nonlocal term with the normal or advection type in nonlocal equations occasionally makes analysis difficult.
To overcome this difficulty, the approximation of a nonlocal term by another type of term can be a solution.
In light of this, we aim to reveal whether advective nonlocal interactions can be approximated by local dynamics.
As a first step, we propose an approximation method for advective nonlocal terms in the nonlocal Fokker--Planck equation using a Keller--Segel system with multiple auxiliary chemotactic factors.
The nonlocal Fokker--Planck equation and Keller--Segel systems are basic models with advective nonlocal interactions and typical local dynamics, respectively.
We show that any smooth kernel can be expanded by combining  the fundamental solutions for an elliptic equation in the Keller--Segel system.
Furthermore, we report that the solution to the nonlocal Fokker--Planck equation with an even smooth kernel can be approximated by that of the Keller--Segel system with specified parameters depending on the integral kernel shape.

The remainder of this paper is organized as follows:
Section \ref{sec:2} outlines the mathematical framework and summarizes the main results. 
In Section \ref{sec:3}, the existence theorem is established, followed by a singular limit analysis detailed in Section \ref{sec:4}. 
Section \ref{sec:cof_aj} starts by presenting a precise formula for the Lagrange interpolation polynomial's coefficient. 
This is followed by a method to ascertain the coefficient of the linear sum in the fundamental solution for an elliptic equation, characterized by the shape of its integral kernel. 
The section also includes a proof of the fundamental solution's series expansion. 
A linear stability analysis is then conducted in Section \ref{sec:LSA}. The paper concludes with Section \ref{sec:7}, summarizing the study's findings and implications.

\section{Mathematical settings and main results}\label{sec:2}
In this section we describe the mathematical settings and results.
We denote the theoretical concentration or cell density at position $x \in \Omega:= [-L , L]$ at time $t>0$ by $\rho = \rho (x,t)$.
We investigate the solution to the following nonlocal Fokker--Planck equation:
\begin{equation}\label{nonlocal:FP}\tag{P}
	\frac{ \partial \rho }{ \partial t }  = \frac{ \partial^2  \rho }{ \partial x^2 }  - \frac{ \partial  }{ \partial x } \Big(  \rho  \frac{ \partial  }{ \partial x } ( W*\rho ) \Big)  \ \text{in} \ \Omega\times (0,\infty),
\end{equation}
where  the periodic boundary condition 
\begin{equation}\label{bc:peri}
\left\{
\begin{aligned}
	&\rho(-L,t) = \rho(L, t), \ t>0, \\
	&\rho_x(-L, t) = \rho_x( L, t), \ t>0,
\end{aligned}
\right.
\end{equation}
is imposed and the initial datum is given by $\rho(x, 0) := \rho_0(x)$. 
Here $W*u(x)$ is defined by
\[
W*u (x) := \int_\Omega W(x-y) u(y) dy
\]
for $W \in L^1_{\mathrm{per}}(\Omega):=\{u|_\Omega \in L^1(\Omega) \ | \ u(x)=u(x+2L), \ x \in \R \}$.
Setting $d_j>0, (j=1,\ldots, M)$ which are the diffusion coefficients in \eqref{eq:system} introduced below,
we define the following function
\begin{equation}\label{fund_sol:v_j}
	k_j(x):=\frac{1}{2\sqrt{d_j} \sinh  \frac{L}{\sqrt{d_j} } } \cosh \frac{L - \left| x \right|}{\sqrt{ d_j } }.
\end{equation}
This is actually a fundamental solution to the elliptic equation explained in Lemma \ref{lemm:k_j} below.
The typical examples of $W$ are as follows:
\begin{align}
	&W(x) =  k_1(x) \text{ with any } d_1>0, \label{kernel:KS}\\
	&W(x) = k_1(x) - k_2(x)  \text{ with any }  d_1< d_2,\label{kernel:lali}\\
	&W(x) = (R_0 - |x| )\chi_{B(R_0)}(x), \label{pot:attract}\\
	&W(x) = ((a_1+a_2)R_0 -a_2R_1 -a_1 |x| )\chi_{B(R_0)}(x) - a_2(R_1-|x|)\chi_{ B(R_1)\backslash B(R_0) }(x) \label{pot:att_rep}
\end{align}
with any $a_1, a_2>0$, where $R_1>R_0>0$ are constants called the sensing radius, $B(R_0)$ is a ball with radius $R_0$ and origin center,  and
\begin{equation*}
	\chi_{B(R_0)}(x)=
	\left\{
\begin{aligned}
	&1 \ & \text{if} \ x \in B(R_0),\\
	&0 \ & \text{otherwise}.
\end{aligned}
\right.\\
\end{equation*}
The profiles of \eqref{kernel:lali} and \eqref{pot:attract} are presented in Figures \ref{fig:Mhat} (a) and \ref{fig4} (a), respectively. 
The nonlocal Fokker--Planck equation \eqref{nonlocal:FP} with the integral kernels \eqref{kernel:KS} and \eqref{kernel:lali} corresponds to the parabolic-elliptic Keller--Segel systems.
A linear stability analysis is presented in Section \ref{sec:LSA}. 
Integral kernels \eqref{pot:attract} and \eqref{pot:att_rep} were introduced by Carrillo et al. \cite{CMSTT2019} and Murakawa and Togashi \cite{MH2015}.
They have a compact support corresponding to the cell body.
If these integral kernels are used for the potential $W$ in \eqref{nonlocal:FP}, 
this describes the situation in which $\rho$ at each point detects the surrounding cell density in the own cell body and the velocity of the aggregation is determined.
This corresponds to Haptotaxis phenomenon.

Firstly, we have the following existence result.
To construct a mild solution to \eqref{nonlocal:FP}
we define a function space with a norm as 
\begin{equation*}
	E:=C([0, \tau], H^1(\Omega)), \quad 
    \left\| \cdot \right\|_{  E } := \left\| \cdot \right\|_{ C([0, \tau],H^1 (\Omega))} 
\end{equation*}
for any time $\tau>0$.
Introducing the following function:
\begin{align}
    G(x,t) := \frac{1 }{2L } \sum_{n \in \Z} e^{-\sigma_n^2 t} e^{i\sigma_n x}, \label{G}
\end{align}
where $i$ is the imaginary unit and 
\[
\sigma_n : = \frac{n \pi}{L}, \quad n \in \Z,
\]
we define the map $\Gamma: E \to E$ as 
\begin{align*}
    \Gamma[u] (x,t) : = (G * \rho_0)(x,t) - \int_0^t  \int_\Omega  G(x-y, t-s) \Big( \frac{\partial}{\partial x} \Big( u  \frac{ \partial  }{ \partial x } W*u  \Big) \Big)(y, s) dy ds, \quad u \in E. \notag
\end{align*}
We say that a function $u \in C([0,T], H^1(\Omega))$ for any $T>0$ is a mild solution to \eqref{nonlocal:FP} with $\rho(x,0)=\rho_0(x)$, provided $ u=\Gamma[u]$.
The following proposition is proven using the standard argument of the fixed point theorem.
\begin{proposition}
Let $R>0$ be an arbitrary  real number, and assume that $W \in W_{\mathrm{per}}^{1,1}(\Omega):=\{u|_\Omega \in W^{1,1}(\Omega) \ | \ u(x)=u(x+2L), \ x \in \R \} $ and 
\begin{align*}
\rho_0 \in H^1(\Omega) \ \text{with} \ \left\| \rho_0 \right\|_{H^1(\Omega)}  < R.
\end{align*}
Then for any $T>0$ there exists a unique mild solution  $\rho$ to \eqref{nonlocal:FP} in $ C( [0, T], H^1(\Omega) ) \cap L^2( 0, T, H^2(\Omega) )$ satisfying
\begin{equation*}
	\left\| \rho \right\|_{C([0, T], H^1 (\Omega))} < C_0,
\end{equation*}
where $C_0$ depends on $R$ and $T$.
Moreover, this mild solution satisfies \eqref{nonlocal:FP} in $L^2(0,T,L^2(\Omega))$.
\end{proposition}

Next, we will approximate the solution to \eqref{nonlocal:FP} with any integral kernel using that to a Keller--Segel system which is a local dynamics.
Introducing the auxiliary factors  $v_j^\eps = v_j^\eps(x,t), \ (j=1, \ldots, M)$, 
we consider the following Keller--Segel system in which the linear sum of $v_j^\eps$ is imposed in nonlocal term in \eqref{nonlocal:FP}:
\begin{equation}\tag{$\mbox{KS}_\eps$} \label{eq:system} 
\left\{
\begin{aligned}
	\frac{ \partial \rho^{\eps}}{ \partial t } & =  \frac{ \partial^2  \rho^{\eps} }{ \partial x^2 }  - \frac{ \partial  }{ \partial x } \Big(  \rho^{\eps}  \frac{ \partial  }{ \partial x } \sum_{j=1}^M a_j v_j^\eps \Big),\\
	\frac{ \partial v_j^\eps }{ \partial t } &= \frac{1}{\eps} \Big( d_j \frac{\partial^2 v_{j}^\eps }{ \partial x^2 }  - v_j^\eps + \rho^{\eps} \Big), \ (j=1,\cdots, M)
\end{aligned}
\right.  \ \text{in} \ \Omega\times (0,\infty).
\end{equation}
Here $0 < \eps \ll 1$ is a sufficiently small parameter, $d_j >0$ is the diffusion coefficient, and each $a_j \in \R$ is a constant that determines whether $v_j^\eps$ acts as an attractive or repulsive substance in the aggregation process of $\rho^\eps$.
Because the solutions to \eqref{eq:system} depends on $\eps$, we denoted them by $(\rho^\eps, v_j^\eps)$, respectively.
The same periodic boundary condition as that in \eqref{nonlocal:FP} is imposed in the equations of \eqref{eq:system} as follows:
\begin{equation*}
\left\{
\begin{aligned}
	&\rho^\eps(-L,t) = \rho^\eps(L, t), \ t>0, \\
	&\rho^\eps_x(-L, t) = \rho^\eps_x( L, t), \ t>0,\\
    &v_j^\eps(-L,t) = v_j^\eps(L,t), \ t>0, \\
    &(v_j^\eps)_x(-L,t) = (v_j^\eps)_x(L,t), \ t>0 \\    
\end{aligned}
\right.  
\end{equation*}
for $j=1, \ldots, M$.
Furthermore, we impose the following initial conditions as 
\begin{equation}\label{IC:KS}
    \rho^\eps (x, 0) :=\rho_0^\eps(x) = \rho _0(x), \quad v_j^\eps(x, 0) := ( v_{j})_0(x), \quad  (j=1, \ldots, M).
\end{equation}

\eqref{eq:system} is a Keller--Segel system with multiple components with the linear sensitivity function.
The role of $v_j^\eps$ can be distinguished by the sign of the coefficient  $a_j$.
If $a_j>0$, then $v_j^\eps$ is an attractive substance for $\rho^\eps$ and $\rho^\eps$ aggregates toward to the region in which the gradient of  $v_j^\eps$ is high independently on the value of its concentration.
In contrast, if $a_j<0$, $v_j^\eps$ acts as a repulsive substance for $\rho^\eps$, and $\rho^\eps$ migrates away from the region in which the gradient of  $v_j^\eps$ is high.

Introducing the following function:
\begin{align}
&G^\eps_j(x,t) := \frac{1 }{2L } \sum_{n \in \Z} e^{ - \frac{ d_j \sigma_n^2 + 1 }{ \eps }t } e^{i\sigma_n x}, \label{G_j}
\end{align}
we define the maps $\Psi_j: E \to E$ and $\Phi: E \to E$ as 
\begin{align*}
	&\Psi_j[u](x,t) : = (G_j^\eps * (v_j)_0)(x,t)
				+ \frac{1}{\eps} \int_0^t \int_\Omega G_j^\eps( x -y, t-s ) u (y, s) dy ds, \quad u \in E, \\
	&\Phi[u] (x,t) : = (G * \rho_0)(x,t) - \int_0^t  \int_\Omega  G(x-y, t-s) \Big( \frac{\partial}{\partial x} \Big( u  \frac{ \partial  }{ \partial x } \sum_{j=1}^M a_j \Psi_j[u]  \Big) \Big)(y, s) dy ds, \quad u \in E, \notag
\end{align*}
respectively.
We now  define the scalar integral equation of \eqref{eq:system} as $\rho^\eps = \Phi[\rho^\eps]$.
We say that a function $(u,\{v_j\}_j) \in C([0,T], H^1(\Omega))\times C( [0, T], C^2(\Omega) )$ for any $T>0$ is a mild solution to \eqref{eq:system} with \eqref{IC:KS}, provided $ u=\Phi[u]$ and $v_j=\Psi_j[v_j]$ for $j=1,\ldots,M$.

With the above settings, we obtain a unique solution to \eqref{eq:system} using the fixed point argument.
\begin{theorem}\label{thm:1}
Let $R>0$ be an arbitrary  real number and assume 
\begin{align}
	& \rho _0\in H^1(\Omega) \ \text{with } \left\| \rho_0 \right\|_{H^1(\Omega)}  < R, \label{ic:rho}\\
	& ( v_{j})_0\in C^2(\Omega), \quad  j=1, \ldots, M. \label{iv:vj}
\end{align}
Then  there exists a real number $\tau_0>0$ such that for any $\eps>0$ there exists a unique mild solution $(\rho^\eps, \{ v_j^\eps \}_{j=1}^M)$ to \eqref{eq:system} in $ C( [0, \tau_0], H^1(\Omega) )  \times C( [0, \tau_0], C^2(\Omega) ) $ satisfying
\begin{equation*}
	\left\| \rho^\eps \right\|_{ C([0, \tau_0], H^1 (\Omega))} <2R.
\end{equation*}
\end{theorem}
We can extend the existence time of the solution to an arbitrary time $T$ as follows:
\begin{corollary}\label{cor:mildsol}
Suppose the same assumption of Theorem \ref{thm:1}.
For any $T>0$, there exists a constant $\tilde {C}_0>0$ depending on $R$ and $T$ but independent of $\eps$ such that there exists a unique mild solution $(\rho^\eps, \{v_j^\eps\}_{j=1}^M)$ to \eqref{eq:system} in $ C( [0, T], H^1(\Omega) ) \cap L^2( 0, T, H^2(\Omega) ) \cap H^1(0,T,L^2(\Omega)) \times C( [0, T], C^2(\Omega) )\cap L^2( 0, T, H^3(\Omega) )$ satisfying
\begin{equation*}
	 \left\| \rho^\eps \right\|_{ C([0,T], H^1 (\Omega))} < \tilde {C}_0.
\end{equation*}
Moreover, this mild solution $(\rho^\eps, \{v_j^\eps\}_{j=1}^M)$ satisfies the system \eqref{eq:system} in $L^2(0,T,L^2(\Omega))\times L^2(0,T,C(\Omega))$.
\end{corollary}

In order to show the relationship between the solutions of nonlocal Fokker--Planck equation \eqref{nonlocal:FP} with any potential $W$ and the Keller--Segel system \eqref{eq:system}, we first investigate the relationship between the solution to \eqref{nonlocal:FP} with the potential provided by $W= \sum_{j=1}^M a_j k_j$ and the solution to \eqref{eq:system}.
\begin{theorem}\label{thm:order_est}
Let $M$ be an arbitrary fixed natural number, and
 $\rho$ be   a solution to \eqref{nonlocal:FP} equipped with $W = \sum_{j=1}^M a_j k_j$ and the initial value  $\rho_0 \in C^2(\Omega)$.
 Let  $\rho^\eps$ be a solution to \eqref{eq:system} equipped with 
\begin{gather}
	 \rho ^\eps _0 =  \rho _0,  \label{iv:rho}\\
	\Big( (v^\eps_{1})_0, \ldots, (v^\eps_{M})_0 \Big) = (  k_1*\rho _0, \ldots,  k_M*\rho _0) .\label{conc_init_data:v_j}
\end{gather}
 Then, for any $\eps>0$ and $T>0$, there exist positive constants $C_1$ and $C_2$ that depend on $a_j$, $d_j$ and $T$, but are independent of $\eps$
 such that 
\begin{align}
 &\left\| \rho^\eps - \rho \right\|_{ C([0,T], H^1(\Omega))} + \left\| \rho^\eps - \rho \right\|_{ L^2(0,T, H^2(\Omega))} \le C_1\eps,\label{ineq:singlar_rho}\\
 &\left\| v_j^\eps - k_j*\rho \right\|_{ C([0,T], H^1(\Omega))} + \left\| v_j^\eps - k_j*\rho \right\|_{ L^2(0,T, H^2(\Omega))}\le C_2 \eps.\label{ineq:singlar_v}
\end{align}
\end{theorem}
\begin{remark}
Using the Sobolev embedding theorem for the first terms on the left-hand sides of \eqref{ineq:singlar_rho} and \eqref{ineq:singlar_v} implies that
    \begin{align*}
 &\left\| \rho^\eps - \rho \right\|_{ C([0,T], C(\Omega))}  \le \tilde{C}_1\eps,\\
 &\left\| v_j^\eps - k_j*\rho \right\|_{ C([0,T], C(\Omega))} \le \tilde{C}_2 \eps,
\end{align*}
where $\tilde{C}_1$ and $\tilde{C}_2$ are obtained  by multiplying the constants of the Sobolev embedding theorem by $C_1$ and $C_2$, respectively.
\end{remark}
The first convergence in Theorem \ref{thm:order_est} shows not only that the solution $\rho^\eps$ to \eqref{eq:system} is sufficiently close to that to \eqref{nonlocal:FP} with $W= \sum_{j=1}^M a_j k_j$ when $\eps$ is very small, but also that the convergence rate is of the order of $\eps$.
The second convergence shows that the solution of the auxiliary substances $v_j^\eps$ is also extremely close to $k_j*\rho$ as $\eps$ tends to $0$.
The proof of this theorem is presented in Section \ref{sec:4}.

Using the convergence of Theorem \ref{thm:order_est}, we can approximate the solution to \eqref{nonlocal:FP} with any even smooth kernel $W$ as that to \eqref{eq:system} with the specified parameters $\{d_j, a_j \}_{j=1}^M$.
Indeed, the parameters $\{ a_j\}_{j=1}^M$ are determined by the shape of $W$ using Theorem \ref{thm:HCSE} in Section \ref{sec:cof_aj}.
Using the interpolation polynomial with the Chebyshev nodes, we can demonstrate the convergence in Theorem \ref{thm:HCSE}.
The explicit formula for the coefficients in the Lagrange interpolation polynomial for an arbitrarily given function is constructed in Proposition \ref{prop:lag_coef} for the proof of Theorem \ref{thm:HCSE}.
Although Theorem \ref{thm:HCSE} and Proposition \ref{prop:lag_coef} are some of our main results, they are presented in Section \ref{sec:cof_aj} for convenience.
Setting the diffusion coefficients of $v_j^\eps$ as 
\begin{equation}\label{set:dj}
	d_j = \frac{1}{(j-1)^2}, \ j=2,\cdots, M
\end{equation}
and defining $d_1$ as a sufficiently large number or limit of infinity, we obtain 
\begin{equation*}
	k_j(x) = \frac{ j-1}{2 \sinh (j-1) L }  \cosh (j-1)(L - \left| x \right|).
\end{equation*}
Note that $k_1=1/(2L)$ if the limit of $d_1\to \infty$ is considered. 
Because the profile of the fundamental solution $k_j$ is unimodal even if the value of $j$ changes, it seems to be difficult to approximate any potential $W$ by the linear sum of  $k_j$.
However, we can obtain the following corollary:
\begin{corollary}\label{cor:w-kj}
Assume that $W \in C^\infty(\Omega)$ and $W$ is even.
    For any $\eps >0$ there exists $M\in \N$, constants $\{a_j\}_{j=1}^M$ and a positive constant $C_{W}$ such that 
    \begin{equation}\label{eq:w-kj}
        \Big\| W - \sum_{j=1}^M a_j k_j \Big\|_{C(\Omega)} \le C_W \eps .
    \end{equation}
\end{corollary}

Because we estimate the error of the solutions of the two nonlocal Fokker--Planck equations with $W=\sum_{j=1}^M a_j k_j$ and any given potential $W(x)$, we prepare the following lemma.
\begin{lemma}\label{lemm:rho1_rho2}
Suppose that $w_1, w_2 \in W_{\mathrm{per}}^{1,1}(\Omega)$ and let $\rho_j, \ (j=1,2)$ denote the solution to
\begin{equation*}\tag{$\mbox{P}_j$} 
\left\{
\begin{aligned}
	&\frac{ \partial \rho_j}{ \partial t } =  \frac{ \partial^2  \rho_j }{ \partial x^2 }  - \frac{ \partial  }{ \partial x } \Big(  \rho_j  \frac{ \partial  }{ \partial x } w_j*\rho_j \Big) \ \text{in} \ \Omega\times (0,\infty),\\
	&\rho_j(x, 0)=\rho_0\in H^1(\Omega),\\
\end{aligned}
\right.
\end{equation*}
respectively.
Then for any $T>0$, there exists a positive constant ${\tilde C_T}$ such that 
\begin{equation}\label{ineq:rho1_rho2}
	\left\| \rho_1  - \rho_2 \right\|_{ C([0,T],L^2 (\Omega))}^2  + \left\| \rho_1 - \rho_2  \right\|_{ L^2 (0,T,H^1(\Omega))}^2 \le {\tilde C_T} \left\| w_1 - w_2 \right\|_{ L^1 (\Omega)}^2.
\end{equation}
\end{lemma}
\begin{remark}\label{rem:rho1_rho2}
    Under the same assumptions, the assertion \eqref{ineq:rho1_rho2} in Lemma \ref{lemm:rho1_rho2} can be improved into
    \begin{equation*}
	\left\| \rho_1  - \rho_2 \right\|_{ C([0,T],H^1 (\Omega))}^2  + \left\| \rho_1 - \rho_2  \right\|_{ L^2 (0,T,H^2(\Omega))}^2 \le {\tilde C_T} \left\| w_1 - w_2 \right\|_{ W^{1,1} (\Omega)}^2.
    \end{equation*}
\end{remark}
This lemma shows that the difference between the solutions to the two Fokker--Planck equations is bounded by the difference between the two potentials.
The proof is presented in Subsection \ref{sec:order}

Referring to $\{ \alpha_j^{M-1}\}$ in Theorem \ref{thm:HCSE}, and putting
\begin{equation}\label{set:aj}
a_{j}=\left\{
    \begin{aligned}
        &2L\alpha_0^{M-1}, \quad (j=1),\\
        &2 \alpha_{j-1}^{M-1} \sinh ((j-1) L)/ (j-1), \quad (j=2,\ldots M),
    \end{aligned}
    \right.
\end{equation}
we estimate the differences between the two solution to \eqref{nonlocal:FP} with an arbitrary even $W$ in $C^\infty([0,L])$ and $W = \sum_{j=1}^M a_j k_j$ by using \eqref{eq:w-kj} in Corollary \ref{cor:w-kj} and \eqref{ineq:rho1_rho2} in Lemma \ref{lemm:rho1_rho2}. 
Moreover, because we can estimate the solutions to \eqref{eq:system} and \eqref{nonlocal:FP} with $W = \sum_{j=1}^M a_j k_j$ from Theorem \ref{thm:order_est}, we obtain the following main result.
\begin{theorem}\label{thm:RDSA}
For any even 2L-periodic function $W$ in $C^\infty([0,L])$, any time $T>0$ and any small positive constant $\eps$, there exist $M \in \N$, a Keller--Segel system \eqref{eq:system} with $M+1$ component, and a positive constant $C_T$ independent of $\eps$ such that 
\begin{equation*}
	 \left\| \rho^\eps - \rho \right\|_{ C([0,T],L^2 (\Omega))} \le C_T \eps,
\end{equation*}
where $\rho$ is the solution to \eqref{nonlocal:FP} equipped with $\rho_0 \in C^2(\Omega)$ and $\rho^\eps$ is the first component of the solution to \eqref{eq:system} equipped with  \eqref{iv:rho} and \eqref{conc_init_data:v_j}.
\end{theorem}
This theorem shows that the solution to nonlocal Fokker--Planck equation with any potential can be approximated by that to the multiple components of the Keller--Segel system with specified parameters.
This convergence result shows a relationship between any advective nonlocal interactions and a local dynamics.
\begin{remark}
Even if $W$ is continuous instead of $C^\infty([0,L])$ in Theorem \ref{thm:RDSA},
the statement in Theorem \ref{thm:RDSA} holds.
However, the parameters $\{ a_j\}_{j=1}^M$ can not be determined for given $W$ although the existence of $\{ a_j\}_{j=1}^M$ is guaranteed in Corollary 1 by Ninomiya et al. \cite{NTY2017}.
\end{remark}

\section{Mild solution}\label{sec:3}
To show Theorem \ref{thm:1}, we present some lemmas.
\subsection{Fundamental solution and boundedness}
First we provide the following lemma:
\begin{lemma}
Functions $G$ of \eqref{G} and $G_j^{\eps}$ of \eqref{G_j} are the fundamental solution for $u_t = u_{xx}$ and 
\begin{equation}\label{eq:heat_v}
u_t =   ( d_j u_{xx}  - u)/ \eps    
\end{equation}
with periodic boundary condition of \eqref{bc:peri} type, respectively.
\end{lemma}
The proof of this lemma is obtained by substituting and calculating.

Using the fundamental solution $G_j^\eps$, we define $\bar{v}_j^\eps$ as
\begin{equation}\label{vj_bar}
\bar{v}_j^\eps = \Psi_j[\phi], \quad \phi \in E.
\end{equation}
We have the following lemma.
\begin{lemma}
Assume $\phi \in E$. Then \eqref{vj_bar} satisfies 
\begin{equation*}
\left\{
\begin{aligned}
	&\frac{ \partial \bar{v}_j^\eps }{ \partial t } = \frac{1}{\eps} \Big( d_j \bar{v}_{j, xx}^\eps  - \bar{v}_j^\eps +\phi \Big),
\quad x\in \Omega, \ t>0, \\
	&\bar{v}_j^\eps(x, 0) =  (v_{j})_0(x), \quad  x\in \Omega.
\end{aligned} 
\right.
\end{equation*}
\end{lemma}
The proof of this lemma is provided by substituting \eqref{vj_bar} into the equation and calculating.
Next we estimate the boundedness of $\bar{v}_j^\eps$ in the following lemma.

\begin{lemma}\label{lem:bound_vj}
Assume that $ \phi \in E$ and \eqref{iv:vj}, and let $C_3$ be a positive constant given by $C_3 : =\sqrt{ L} /(d_j \sqrt{6})$.
Then it holds that
\begin{align}
	& \left\| ( \bar{v}_{j}^\eps)_{x} \right\|_{C([0, \tau], C(\Omega))} 
	\le \left\|  ( v_{j})_{0,x}  \right\|_{ C (\Omega)}
	+ C_3 \left\| \phi \right\|_{ C([0, \tau], L^2 (\Omega) ) } \label{boun:v_j,x}, \\
	& \left\| ( \bar{v}_{j}^\eps)_{xx} \right\|_{C([0, \tau], C(\Omega))} 
	\le \left\|  ( v_{j})_{0,xx}  \right\|_{ C (\Omega)}
	+ C_3  \left\| \phi_x \right\|_{ C([0, \tau],  L^2 (\Omega))  } \label{boun:v_j,xx}.
\end{align}
\end{lemma}
\begin{proof}[Proof of Lemma \ref{lem:bound_vj}]
From the triangular inequality, we have
\begin{equation}\label{v_j,x^eps:1and2}
\begin{aligned}
	 \left\| (\bar{v}_{j}^\eps)_{x} \right\|_{C([0, \tau], C(\Omega))}  
	&\le \left\|  G_j^\eps*( v_{j})_{0,x} \right\|_{C([0, \tau], C (\Omega))} \\
	&\quad + \frac{1}{\eps}  \left\|  \int_0^t \int_{\Omega}  (G_{j}^\eps)_x (\cdot-y, \cdot-s) \phi (y,s) dy ds \right\|_{ C([0, \tau], C (\Omega))}.
\end{aligned}
\end{equation}
By the maximum principle for the heat equation \eqref{eq:heat_v},
we compute the first term.
\begin{align*}
\left\|   G_j^\eps*( v_{j})_{0,x}  \right\|_{C([0, \tau], C (\Omega))}
&\le  \left\|  \int_{\Omega}  G_j^\eps (\cdot-y, 0)( v_{j})_{0,x} (y)  dy  \right\|_{ C (\Omega)}\\
&=  \left\|  ( v_{j})_{0,x}  \right\|_{ C (\Omega)}.
\end{align*}
Next, we denote the Fourier coefficient of $\phi$ by 
\begin{equation*}
	p_n(t) := \frac{1}{\sqrt{2L}}\int_\Omega \phi (x,t) e^{-i\sigma_n x} dx.
\end{equation*}
Before estimating the second term of \eqref{v_j,x^eps:1and2}, we can compute that 
\begin{align*}
	&\sum_{n \in \Z}  \Big( \frac{ | \sigma_n| }{d_j\sigma_n^2 + 1}  \Big)^2
	=\sum_{n \ne 0}  \Big( \frac{ | \sigma_n| }{d_j\sigma_n^2 + 1}  \Big)^2
	= 2 \sum_{n = 1}^\infty  \Big( \frac{  \sigma_n }{d_j\sigma_n^2 + 1}  \Big)^2 \\
	&\le 2 \sum_{n = 1}^\infty  \Big( \frac{  1 }{d_j\sigma_n }  \Big)^2 
	 = \frac{2}{d_j^2} \frac{ L^2 }{\pi^2}  \sum_{n = 1}^\infty \frac{1}{n^2}
	 = \frac{ L^2 }{3 d_j^2 }.
\end{align*}
Then we see that 
\begin{align*}
	&\frac{1}{\eps}  \left\|  \int_0^t \int_{\Omega}  (G_{j}^\eps)_x (\cdot-y, \cdot-s) \phi (y,s) dy ds \right\|_{C([0, \tau], C (\Omega))} \\
	&=  \frac{1}{\eps} \sup_{\substack{ t\in[0, \tau], \\ x \in \Omega }} \left| \frac{1}{2L}  \sum_{n \in \Z}  i\sigma_n e^{ - \frac{ d_j\sigma_n^2 + 1 }{ \eps }t  } 
	 			e^{ i\sigma_n x }  \int_0^t e^{  \frac{ d_j\sigma_n^2 + 1 }{ \eps }s  }  \int_\Omega e^{- i \sigma_n y} \phi(y,s)  dy ds \right| \\
	&=  \frac{1}{\eps} \sup_{\substack{ t\in[0, \tau], \\ x \in \Omega }} \left| \frac{1}{\sqrt{2L}}  \sum_{n \in \Z} i\sigma_n  e^{ - \frac{ d_j\sigma_n^2 + 1 }{ \eps }t  } 
	 			e^{ i\sigma_n x }  \int_0^t e^{  \frac{ d_j\sigma_n^2 + 1 }{ \eps }s  } p_n(s) ds \right| \\	
	&\le \frac{1}{\sqrt{2L}}  \sup_{\substack{ t\in[0, \tau], \\ x \in \Omega }}  \left|  \sum_{n \in \Z} 
	 			 i\sigma_n e^{ i\sigma_n x }   \frac{ 1 }{d_j\sigma_n^2 + 1} \Big( 1 - e^{ - \frac{ d_j\sigma_n^2 + 1 }{ \eps }t } \Big) \sup_{s \in [0,t]} | p_n(s) |  \right| \notag\\	
	&\le \frac{1}{\sqrt{2L}} \sup_{ t\in[0, \tau] }   \sum_{n \in \Z} \frac{ | \sigma_n | }{d_j\sigma_n^2 + 1}  \Big( 1 - e^{ - \frac{ d_j\sigma_n^2 + 1 }{ \eps }t } \Big)   | p_n(t) |  
	\le \frac{1}{\sqrt{2L}}   \sum_{n \in \Z}   \frac{ | \sigma_n | }{d_j\sigma_n^2 + 1} \sup_{ t\in[0, \tau] }  | p_n(t) |  \\
	& \le \frac{1}{\sqrt{2L}}  \sqrt{ \sum_{n \in \Z}  \Big( \frac{ | \sigma_n| }{d_j\sigma_n^2 + 1}  \Big)^2 } \sup_{ t\in[0, \tau] }\sqrt{ \sum_{n \in \Z}  | p_n(t) |^2  } \\
	&\le C_3 \left\| \phi \right\|_{ C( [0, \tau], L^2 (\Omega))}.
\end{align*}
Thus, we obtain the first estimation. 
Replacing the functions $(v_{j})_{0,x}$ and $ \phi $ with $(v_{j})_{0,xx}$ and $ \phi_{x} $ in \eqref{v_j,x^eps:1and2}, respectively,
we have also the second assertion of the boundedness by the  same calculation as above.
\end{proof}

Because we will use the following boundedness several times, we give the following lemma:
\begin{lemma}\label{lemm:bounded_G}
Assume   $f=f(x) \in L^2(\Omega)$ and $ g = g(x,t)\in C( (0,T],L^2(\Omega))$ for any $T>0$, and
let $C_4$ be a positive constant given by	$C_4:=\pi/(d_jL)$.
Then we obtain that, for all $t\in (0,T]$,  
\begin{align}
&\left\| G(\cdot, t)*f \right\|_{ L^2 (\Omega)} 
	\le \left\| f \right\|_{ C((0,T], L^2 (\Omega) )}, \label{b:Gg}\\
&\left\|  \int_0^t  \int_\Omega  G( \cdot -y , t-s)  g(y,s)  dy ds \right\|_{ L^2 (\Omega)} 
	\le t \left\| g \right\|_{ C([0,T], L^2 (\Omega)) }, \label{b:Gf}\\
&\left\|  \int_0^t  \int_\Omega  G_x( \cdot -y , t-s)  g(y,s)  dy ds \right\|_{ L^2 (\Omega)} 
	\le \sqrt{t} \left\| g \right\|_{ C([0,T], L^2 (\Omega)) }, \label{b:Gxf}\\
&\frac{1}{\eps} \left\|  \int_0^t  \int_\Omega  G_j^\eps( \cdot -y , t-s)  g(y,s)  dy ds \right\|_{ L^2 (\Omega)}
	 \le  \left\| g \right\|_{ C([0,T], L^2 (\Omega)) }, \label{b:Gjf}\\
&\frac{1}{\eps} \left\|  \int_0^t  \int_\Omega  (G_{j}^\eps)_x( \cdot -y , t-s)  g(y,s)  dy ds \right\|_{ L^2 (\Omega)}
 	\le C_4 \left\| g \right\|_{ C([0,T], L^2 (\Omega)) }. \label{b:Gjxf}
\end{align}
\end{lemma}
We note that  the right-hand sides of \eqref{b:Gjf} and \eqref{b:Gjxf} do not depend on $\eps$.
The proof is based on the estimation of the Fourier coefficients of $f$ and $g$.
As this is a straightforward calculation, we present the proof in in the \ref{append:LBG}.

\subsection{Contraction map}
We show that the map $\Phi$ becomes a contraction map from 
\begin{equation*}
	B_R := \{ v \in E \ | \ \left\| v \right\|_{  E } < 2R  \}
\end{equation*}
to $B_R$ taking the sufficiently small time $\tau > 0$. 
\begin{lemma}\label{lem:MR}
Assume that $\phi \in B_R$ and \eqref{iv:vj}. 
Then the following estimate holds
\begin{equation*}
	\sup_{ \left\| \phi \right\|_{  E} <2R }
	 \left\| \frac{ \partial  }{ \partial x  }  ( \phi \sum_{j=1}^M a_j ( \bar{v}_{j}^\eps)_x )  \right\|_{ L^2 (\Omega)} 
	\le M_R,
\end{equation*}
where 
\begin{equation*}
	M_R : = 2R \sum_{j=1}^M | a_j |  \Big( \left\|( v_{j})_{0,x} \right\|_{ C(\Omega)}  + \left\|( v_{j})_{0,xx} \right\|_{ C(\Omega)}   +4C_3R \Big).
\end{equation*}
\end{lemma}
\begin{proof}[Proof of Lemma \ref{lem:MR}]
Using the Minkowski inequality and Lemma \ref{lem:bound_vj}, we see that
\begin{align*}
 	& \left\| \frac{ \partial  }{ \partial x  }  ( \phi \sum_{j=1}^M a_j ( \bar{v}_{j}^\eps)_x )  \right\|_{ L^2 (\Omega)} (t)
 	\le \sum_{j=1}^M | a_j |  \left\| \frac{ \partial  }{ \partial x  }  ( \phi ( \bar{v}_{j}^\eps)_x )  \right\|_{ L^2 (\Omega)}(t) 
	=  \sum_{j=1}^M | a_j |  \left\|  \phi_x ( \bar{v}_{j}^\eps)_x  + \phi ( \bar{v}_{j}^\eps)_{xx}   \right\|_{ L^2 (\Omega)}(t) \\
	& \le \sum_{j=1}^M | a_j | \Big(  \left\| \phi_x ( \bar{v}_{j}^\eps)_x  \right\|_{ L^2 (\Omega)} (t) +  \left\| \phi ( \bar{v}_{j}^\eps)_{xx} \right\|_{ L^2 (\Omega)} (t)  \Big) \\
	& \le \sum_{j=1}^M | a_j | \Big(  \left\| ( \bar{v}_{j}^\eps)_x  \right\|_{  C(\Omega)} (t) \left\| \phi_x \right\|_{ L^2 (\Omega)} (t)  
						     +  \left\|  ( \bar{v}_{j}^\eps)_{xx} \right\|_{  C(\Omega)} (t) \left\| \phi \right\|_{ L^2 (\Omega)}  (t) \Big)\\
	& < \sum_{j=1}^M | a_j | \Big\{  \Big( \left\|  ( v_{j})_{0,x}  \right\|_{ C (\Omega)}
	+ C_3 \left\| \phi \right\|_{ C( [0,\tau ], L^2 (\Omega) ) } \Big)
	 \left\| \phi_x \right\|_{ L^2 (\Omega)}(t)    \\
	& \qquad\quad	\quad   + \Big(  \left\|  ( v_{j})_{0,xx}  \right\|_{ C (\Omega)}
	+ C_3 \left\| \phi_x \right\|_{ C( [0,\tau ], L^2 (\Omega) ) } \Big) \left\| \phi \right\|_{ L^2 (\Omega)}(t)  \Big\}\\
	&\le M_R.
\end{align*}
\end{proof}

Next, we have the following lemma.
\begin{lemma}\label{lem:Phi_MR}
Assume that $\phi \in B_R$ and that \eqref{ic:rho}.
Then
\begin{equation*}
	\left\| \Phi[ \phi ] \right\|_{ H^1 (\Omega)}(t) \le R +tM_R +\sqrt{ t } M_R
\end{equation*}
holds.
\end{lemma}
\begin{proof}[Proof of Lemma \ref{lem:Phi_MR}]
From the Minkowski inequality we see that 
\begin{align}
	\left\| \Phi[ \phi ] \right\|_{ H^1 (\Omega)} (t) \notag
    &\le \left\| G * \rho_0 \right\| _{ H^1 (\Omega)}  (t)\\
	& \qquad +\left\|  \int_0^t  \int_\Omega  G( \cdot -y , t-s) \Big( \frac{\partial}{\partial x} \Big( \phi  \frac{ \partial  }{ \partial x } \sum_{j=1}^M a_j \bar{v}_j^\eps  \Big) \Big)(y,s) dy ds  \right\| _{ H^1 (\Omega)} . \label{eq:phi[rho]}
\end{align}
Then using  \eqref{b:Gg} in Lemma \ref{lemm:bounded_G},  we can compute the first term of \eqref{eq:phi[rho]} as follows:
\begin{align*}
	\left\| G * \rho_0 \right\| _{ H^1 (\Omega)} ^2 (t)
	&= \left\| G * \rho_0 \right\|_{  L^2 (\Omega)} ^2 (t)
	+  \left\| G * (\rho_{0})_x \right\|_{  L^2 (\Omega)} ^2 (t) \\
	&= \left\| \rho_0 \right\|_{ L^2 (\Omega)}^2+\left\| (\rho_{0})_x \right\|_{ L^2 (\Omega)}^2 
	=\left\| \rho_0 \right\|_{ H^1 (\Omega)} \\
	&< R^2.
\end{align*}

Next, we estimate the second term of \eqref{eq:phi[rho]}.
Utilizing \eqref{b:Gf} in Lemmas \ref{lemm:bounded_G} and \ref{lem:MR}, we compute that
\begin{align*}
&\left\|  \int_0^t  \int_\Omega  G( \cdot -y , t-s) \Big( \frac{\partial}{\partial x} \Big( \phi  \frac{ \partial  }{ \partial x } \sum_{j=1}^M a_j \bar{v}_j^\eps  \Big) \Big) (y,s)dy ds  \right\| _{ L^2 (\Omega)}^2 \\
& \le t^2  \left\|  \frac{\partial }{ \partial  x } ( \phi \sum_{j=1}^M a_j (\bar{v}_{j}^\eps)_x  )  \right\|_{ C([0,\tau], L^2 (\Omega))}^2 
 \le t^2 M_R^2.
\end{align*}
Similarly, \eqref{b:Gxf} in Lemma \ref{lemm:bounded_G} yields that
\begin{align*}
&\left\|  \int_0^t  \int_\Omega  G_x( \cdot -y , t-s) \Big( \frac{\partial}{\partial x} \Big( \phi  \frac{ \partial  }{ \partial x } \sum_{j=1}^M a_j \bar{v}_j^\eps  \Big) \Big)(y,s) dy ds  \right\| _{ L^2 (\Omega)}^2 \\
& \le t  \left\|  \frac{\partial }{ \partial  x } ( \phi \sum_{j=1}^M a_j ( \bar{v}_{j}^\eps)_x  )  \right\|_{C([0,\tau], L^2 (\Omega))}^2
 \le tM_R^2.
\end{align*}

\end{proof}
Consequently, we observe that $\Phi: E \to E$, and
that there exists $\tau_1 >0$ independent of $\eps$ such that for $\tau < \tau_1$ $\Phi$ is a map from $B_R$ to $B_R$.

Next, we set 
\begin{equation*}
	\widetilde{v}_j^\eps (x, t) : = \Psi_j(\psi)
, \quad \psi \in E.
\end{equation*}

\begin{lemma}\label{lem:dif_v_j}
Assume that $\phi, \psi \in E$.
Then it holds that 
\begin{align}
	&\left\| (\bar{v}_{j}^\eps)_x - (\widetilde{ v }_{j}^\eps)_x  \right\|_{ C([0, \tau], L^2 (\Omega))} 
	\le \left\| \phi_x - \psi_x  \right\|_{ C([0, \tau],L^2 (\Omega) )}, \label{v_tilde:H^1} \\
	& \left\| (\bar{v}_{j}^\eps)_{xx} - (\widetilde{ v}_{j}^\eps)_{xx} \right\|_{ C([0, \tau], L^2 (\Omega))} 
	\le C_4 \left\| \phi_x - \psi_x \right\|_{ C([0, \tau], L^2 (\Omega) ) }, \label{v_tilde:H^2}
\end{align}
where $C_4$ denotes the constant defined in Lemma \ref{lemm:bounded_G}.
\end{lemma}
\begin{proof}[Proof of Lemma \ref{lem:dif_v_j}]
Using \eqref{b:Gjf} in Lemma \ref{lemm:bounded_G},
We compute that 
\begin{align*}
	\left\| ( \bar{v}_{j}^\eps)_x - (\widetilde{ v }_{j}^\eps)_x  \right\|_{ L^2 (\Omega)} ^2(t)
	&= \frac{1}{\eps^2} \left\| \int_0^t  \int_\Omega  G_j^\eps( \cdot -y , t-s) ( \phi_x - \psi_x ) (y,s) dy ds \right\|_{ L^2 (\Omega)}^2  \\
	& \le \left\| \phi_x - \psi_x  \right\|_{ C([0, \tau], L^2 (\Omega))}^2.
\end{align*}
Similarly,  \eqref{b:Gjxf} in Lemma \ref{lemm:bounded_G} shows that 
\begin{align*}
	\left\| ( \bar{v}_{j}^\eps)_{xx} - (\widetilde{ v }_{j}^\eps)_{xx}  \right\|_{ L^2 (\Omega)} ^2 (t)
	&= \frac{1}{\eps^2} \left\| \int_0^t  \int_\Omega  (G_{j}^\eps)_x( \cdot -y , t-s) ( \phi_x - \psi_x )  (y,s)dy ds \right\|_{ L^2 (\Omega)}^2 \\
	& \le C_4 \left\| \phi_x - \psi_x  \right\|_{ C([0, \tau], L^2 (\Omega))}^2.
\end{align*}
\end{proof}
Finally, we show that map $\Phi$ becomes a contraction map by taking a sufficiently small $\tau_2>0$ and setting $\tau = \tau_2$.
\begin{lemma}\label{lem:cont_map}
Assume that $\phi, \psi \in E$.
Then there exists a positive constant $C_5$ independent of $\eps$ such that 
\begin{equation*}
	\left\| \Phi[\phi] - \Phi[\psi]  \right\|_{ H^1 (\Omega)} (t)
	\le C_5  \sqrt{t}  \left\| \phi - \psi \right\|_{ C([0,\tau], H^1 (\Omega))}.
\end{equation*}

\end{lemma}
\begin{proof}[Proof of Lemma \ref{lem:cont_map}]
Since
\begin{equation}\label{terms:Phi}
	\left\| \Phi[\phi] - \Phi[\psi]  \right\|_{ H^1 (\Omega)} ^2 (t)
	= \left\| \Phi[\phi] - \Phi[\psi]  \right\|_{ L^2 (\Omega)} ^2 (t)
	+ \left\| \frac{ \partial }{ \partial x} ( \Phi[\phi] - \Phi[\psi] ) \right\|_{ L^2 (\Omega)} ^2 (t),
\end{equation}
we estimate each term on right-hand side.
We can compute that  
\begin{align*}
	\left\| \Phi[\phi] - \Phi[\psi]  \right\|_{ L^2 (\Omega)} (t)
	&= \left\| \int_0^t \int_\Omega G( \cdot - y, t-s ) \frac{ \partial }{\partial x } 
			\Big( \phi \sum_{j=1}^M a_j ( \bar{v}_{j}^\eps)_x -  \psi \sum_{j=1}^M a_j (\widetilde{v}_{j}^\eps)_x \Big) (y,s) dy ds \right\|_{ L^2 (\Omega)} \\
	&\le \left\| \int_0^t \int_\Omega G( \cdot - y, t-s ) \frac{ \partial }{\partial x } 
			\Big( \phi \sum_{j=1}^M a_j  ( ( \bar{v}_{j}^\eps)_x -  (\widetilde{v}_{j}^\eps)_x ) \Big) (y,s) dy ds \right\|_{ L^2 (\Omega)} \\
	&\quad + \left\| \int_0^t \int_\Omega G( \cdot - y, t-s ) \frac{ \partial }{\partial x } 
			\Big( ( \phi  -  \psi  )\sum_{j=1}^M a_j (\widetilde{v}_{j}^\eps)_x \Big) (y,s) dy ds \right\|_{ L^2 (\Omega)} \\
	&=: K_1(t) + K_2(t).
\end{align*}
Using \eqref{b:Gxf} in Lemma \ref{lemm:bounded_G}, we estimate that 
\begin{align*}
	K_1(t)&=  \left\| \int_0^t  \int_\Omega  G_x( \cdot -y , t-s)\Big( \phi \sum_{j=1}^M a_j  ( ( \bar{v}_{j}^\eps)_x -  (\widetilde{v}_{j}^\eps)_x )  \Big) (y,s)  dy ds \right\|_{ L^2 (\Omega)}  \\
	&\le  \sqrt{ t} \left\| \phi \sum_{j=1}^M a_j  ( ( \bar{v}_{j}^\eps)_x -  (\widetilde{v}_{j}^\eps)_x )  \right\|_{ C( [0, \tau], L^2 (\Omega))} \\
	& \le \sqrt{ t}  \left\| \phi \right\|_{ C([0,\tau], C (\Omega))} \left\|  \sum_{j=1}^M a_j  ( ( \bar{v}_{j}^\eps)_x -  (\widetilde{v}_{j}^\eps)_x )   \right\|_{ C( [0, \tau], L^2 (\Omega))} \\
	& \le \sqrt{ t} \left\| \phi \right\|_{ C([0,\tau], C (\Omega))} \sum_{j=1}^M |a_j|   \left\|  ( \bar{v}_{j}^\eps)_x -  (\widetilde{v}_{j}^\eps)_x  \right\|_{ C( [0, \tau], L^2 (\Omega))} \\
	& \le \sqrt{ t}  \left\| \phi \right\|_{ C([0,\tau], C (\Omega))} \sum_{j=1}^M |a_j|   \left\|  \phi_x - \psi_x  \right\|_{ C( [0, \tau], L^2 (\Omega))},
\end{align*}
where we used  $\phi \in C(\Omega\times [0,\tau])$ from the Sobolev embedding theorem, the Minkowski inequality, and \eqref{v_tilde:H^1} in Lemma \ref{lem:dif_v_j}.

Similarly to this estimation, \eqref{b:Gf} in Lemma \ref{lemm:bounded_G} yields that
\begin{align*}
K_2(t) &= \left\| \int_0^t G( \cdot - y, t-s ) \frac{ \partial }{\partial x } \Big( ( \phi  -  \psi  )\sum_{j=1}^M a_j (\widetilde{v}_{j}^\eps)_x  \Big) (y,s)  dy ds \right\|_{ L^2 (\Omega)}  \\
& \le t  \left\| \frac{ \partial }{\partial x }  \Big( ( \phi  -  \psi  )\sum_{j=1}^M a_j (\widetilde{v}_{j}^\eps)_x  \Big) \right\|_{ C( [0, \tau], L^2 (\Omega))} \\
& \le  t  \left( \left\|  \Big( ( \phi_x  -  \psi_x  )\sum_{j=1}^M a_j (\widetilde{v}_{j}^\eps)_x  \Big) \right\|_{ C( [0, \tau], L^2 (\Omega))} 
	+ \left\|  \Big( ( \phi  -  \psi  )\sum_{j=1}^M a_j (\widetilde{v}_{j}^\eps)_{xx}  \Big) \right\|_{ C( [0, \tau], L^2 (\Omega))}  \right) \\
&\le t  \Big( \sum_{j=1}^M | a_j | \left\| (\widetilde{v}_{j}^\eps)_x \right\|_{ C([0,\tau], C (\Omega))} \left\| \phi_x  -  \psi_x \right\|_{ C( [0, \tau],  L^2 (\Omega))} 
	+  \sum_{j=1}^M | a_j | \left\| (\widetilde{v}_{j}^\eps)_{xx} \right\|_{ C([0,\tau], C (\Omega))}  \left\| \phi  -  \psi \right\|_{  C( [0, \tau], L^2 (\Omega))}  \Big)\\
&\le t C_6  \left\| \phi - \psi \right\|_{ C( [0, \tau], H^1 (\Omega))},
\end{align*}
where we utilized the Minkowski inequality, the boundedness  \eqref{boun:v_j,x} and \eqref{boun:v_j,xx} in Lemma \ref{lem:bound_vj}, and 
\begin{equation}\label{const:max}
	C_6:= \max\{ \sum_{j=1}^M | a_j | \left\| (\widetilde{v}_{j}^\eps)_x \right\|_{ C([0,\tau],C (\Omega))},  \sum_{j=1}^M | a_j | \left\| (\widetilde{v}_{j}^\eps)_{xx} \right\|_{ C([0,\tau],C (\Omega))} \}.
\end{equation}
We note that $C_6$ does not depend on $\eps$.

Next, we estimate the second term of \eqref{terms:Phi} on right-hand side.
First, we write as
\begin{align*}
 \left\| \frac{ \partial }{ \partial x} ( \Phi[\phi] - \Phi[\psi] ) \right\|_{ L^2 (\Omega)} (t)
& \le \left\|  \int_0^t \int_\Omega G_x( \cdot-y, t-s ) \frac{\partial}{\partial x} \Big(  \phi \sum_{j=1}^M a_j  ( ( \bar{v}_{j}^\eps)_x -  (\widetilde{v}_{j}^\eps)_x ) \Big)(y,s) dy ds \right\|_{ L^2 (\Omega)}  \\
 &+  \left\|  \int_0^t \int_\Omega G_x( \cdot-y, t-s ) \frac{\partial}{\partial x} \Big(   ( \phi  -  \psi  )\sum_{j=1}^M a_j (\widetilde{v}_{j}^\eps)_x \Big)(y,s) dy ds \right\|_{ L^2 (\Omega)}  \\
 &=: \mathcal{K}_1(t) +  \mathcal{K}_2(t).
\end{align*}
Similarly to the previous estimations, using  \eqref{b:Gxf} in Lemma \ref{lemm:bounded_G}, we obtain that 
\begin{align*}
\mathcal{K}_1(t)
    &\le \sqrt{t} \Big(  \left\|  \phi_x \sum_{j=1}^M a_j  ( ( \bar{v}_{j}^\eps)_x -  (\widetilde{v}_{j}^\eps)_x )\right\|_{ C( [0, \tau], L^2 (\Omega))}
	+  \left\|  \phi \sum_{j=1}^M a_j  ( ( \bar{v}_{j}^\eps)_{xx} -  (\widetilde{v}_{j}^\eps)_{xx} )\right\|_{ C( [0, \tau], L^2 (\Omega))}  \Big) \\
	& \le \sqrt{t}  \Big( \sum_{j=1}^M | a_j | \left\| ( \bar{v}_{j}^\eps)_x - (\widetilde{v}_{j}^\eps)_x \right\|_{C( [0, \tau],  C (\Omega))} \left\| \phi_x \right\|_{C( [0, \tau],  L^2 (\Omega))} \\
	& \qquad \qquad  +  \sum_{j=1}^M | a_j | \left\| \phi \right\|_{ C( [0, \tau], C (\Omega))}  \left\| ( \bar{v}_{j}^\eps)_{xx} - (\widetilde{v}_{j}^\eps)_{xx}\right\|_{C( [0, \tau],  L^2 (\Omega)} ) \Big)\\
	& \le  \sqrt{t}  \sum_{j=1}^M | a_j |  \Big( C_3 \left\| \phi_x \right\|_{ C( [0, \tau], L^2 (\Omega))}  \left\| \phi - \psi \right\|_{ C( [0, \tau], L^2 (\Omega))}  
		+C_4 \left\| \phi \right\|_{C( [0, \tau],  C (\Omega))} \left\| \phi_x - \psi_x \right\|_{ C( [0, \tau], L^2 (\Omega))}\Big)\\
	&\le \sqrt{t} C_7 \left\| \phi - \psi  \right\|_{C( [0, \tau],  H^1 (\Omega))},
\end{align*}
where we used the boundedness \eqref{boun:v_j,x} in Lemma \ref{lem:bound_vj}, the boundedness that $\phi \in C(\Omega)$,  \eqref{v_tilde:H^1} and \eqref{v_tilde:H^2} in Lemma \ref{lem:dif_v_j}, and we put
\begin{equation*}
	C_7 := \max\{ \left\| \phi_x \right\|_{C([0,\tau], L^2 (\Omega))} \Big(  \sum_{j=1}^M | a_j |C_3 \Big), \  \left\| \phi\right\|_{ C([0,\tau], L^2 (\Omega)) }  \Big( \sum_{j=1}^M | a_j | C_4 \Big) \}.
\end{equation*}
It should also be noted that $C_7$ does not depend on $\eps$.

Finally,  using  \eqref{b:Gxf} in Lemma \ref{lemm:bounded_G}, we obtain that 
\begin{align*}
	\mathcal{K}_2(t)
     &\le \sqrt{t} \Big( \left\| ( \phi_x  -  \psi_x  )\sum_{j=1}^M a_j (\widetilde{v}_{j}^\eps)_x  \right\|_{ C( [0, \tau], L^2 (\Omega))} 
		+ \left\| ( \phi  -  \psi  )\sum_{j=1}^M a_j (\widetilde{v}_{j}^\eps)_{xx}  \right\|_{ C( [0, \tau], L^2 (\Omega))}  \Big)\\
		&\le \sqrt{t} \sum_{j=1}^M |a_j| \Big(  \left\| (\widetilde{v}_{j}^\eps)_x  \right\|_{ C( [0, \tau],  C (\Omega))}   \left\| \phi_x  -  \psi_x  \right\|_{ C( [0, \tau], L^2 (\Omega))} 
			+ \left\| (\widetilde{v}_{j}^\eps)_{xx}  \right\|_{C( [0, \tau],  C (\Omega))}   \left\| \phi  -  \psi  \right\|_{ C( [0, \tau], L^2 (\Omega))} \Big)\\
		&\le \sqrt{t}  C_6  \left\| \phi - \psi \right\|_{C( [0, \tau],  H^1 (\Omega))},
\end{align*}
where $C_6$ is as defined in \eqref{const:max}.
Putting \[
C_5:=\left\| \phi \right\|_{ C([0,\tau], C (\Omega))} \sum_{j=1}^M |a_j| +C_6\sqrt{\tau}+C_6+C_7, \]
we complete the proof.
\end{proof}

Consequently, taking a sufficiently small value $\tau=\tau_2>0$ which is independent of $\eps$, we obtain 
\begin{equation*}
	\left\| \Phi[\phi] - \Phi[\psi] \right\|_{ C([0, \tau_2],H^1 (\Omega))} \le \frac 1 2 \left\| \phi - \psi \right\|_{ C([0,\tau_2], H^1 (\Omega))}.
\end{equation*}
Thus, the map $\Phi: B_R \to B_R$ is a contraction map.

\begin{proof}[Proof of Theorem \ref{thm:1}]
By setting $\tau_0 := \min\{ \tau_1, \tau_2\}$, we see that the map $\Phi: B_R \to B_R$ is a contraction map.
From the Banach fixed-point theorem, the equation $\rho^\eps = \Phi[\rho^\eps]$ has a unique solution in $C([0,\tau_0], H^1(\Omega))$.

\end{proof}

\begin{proof}[Proof of Corollary \ref{cor:mildsol}]
Repeating to use Theorem \ref{thm:1}, we can connect the mild solutions for $t \in [0,T]$ at any time $T>0$.
 Using the term by term of the weak derivative for the integral equations with respect to $x$ and $t$, we observe that the solutions $\rho^\eps$ and $v_j^\eps$ satisfy \eqref{eq:system} in $L^2(0,T,L^2(\Omega))$ and $L^2(0,T,C(\Omega))$, respectively.

Differentiating the right-hand sides of $\rho^\eps = \Phi[\rho^\eps]$ and $v_j^\eps = \Psi_j[v_j^\eps]$ with respect to $x$ by applying the term by term of the weak derivative, respectively, 
we can show that $\rho^\eps \in L^2( 0, T, H^2(\Omega) )$ and $v_j^\eps \in L^2( 0, T, H^3(\Omega) )$.
Similarly, differentiating the right-hand side of $\rho^\eps = \Phi[\rho^\eps]$ with respect to $t$ by applying the term by term of the weak derivative,
we see that $\rho^\eps \in H^1(0,T, L^2(\Omega))$.
\end{proof}

\section{Singular limit analysis}\label{sec:4}
To demonstrate Theorem \ref{thm:order_est} we prepare the lemmas in the following subsections.
\subsection{Fundamental solution} 
First we have the following lemma.
\begin{lemma}\label{lemm:k_j}
$k_j$ in \eqref{fund_sol:v_j} is the fundamental solution to 
\begin{equation*}
\left\{
\begin{aligned}
	&-d_jv_{xx} + v = \delta(x),\\
	&v(-L) = v(L),\\
	&v_{x}(-L) = v_{x}(L),
\end{aligned}
\right.
\end{equation*}
where $\delta$ denotes the Dirac delta function.
\end{lemma}
The proof of this lemma is provided by the substitution of $k_j$ and calculations.

Setting a constant in $k_j$ as 
\begin{equation*}
	c_k(j) : = \frac{1}{2\sqrt{d_j} \sinh ( L / \sqrt{d_j}) }, \quad j=1,\ldots, M,
\end{equation*}
we have the following lemma.
\begin{lemma} \label{lemm:wd_kj}
	For $d_j>0$ and $j=1,\ldots,M$,
	\begin{align*}
		& (k_j)_x(x) = 
		\left\{
		\begin{aligned}
		&\frac{ -c_k(j) }{\sqrt{d_j}} \sinh \frac{ L-x }{\sqrt{d_j}},\quad x \in (0,L],\\
		&\frac{ c_k(j) }{\sqrt{d_j}} \sinh \frac{ L+x }{\sqrt{d_j}},\quad x \in [-L,0),
		\end{aligned}\right.\\
		& (k_j)_{xx}(x) 
  = \frac{1}{d_j} k_j(x), \quad x \in \Omega,\\
		& (k_j)_{xxx}(x) =  \frac{1}{d_j} (k_j)_x(x), 
	\end{align*}
	hold in the weak sense.
\end{lemma}
For the proof of the Lemma \ref{lemm:wd_kj},
we can directly obtain the weak derivatives
by multiplying $k_j$ by the test function $\varphi \in C_0^\infty(\Omega)$, and the integration by parts.

\begin{lemma}\label{lemm:k_j_L^1}
Let $C_8$ be a positive constant given by
\begin{equation*}
	 C_8 := 2c_k(j) \Big( \cosh \frac{L}{ \sqrt{ d_j} } - 1 \Big).
\end{equation*}
Then it holds that
\begin{align*}
&\left\|  k_j \right\|_{ L^1 (\Omega)} = 1,\\
&\left\|  (k_j)_x \right\|_{ L^1 (\Omega)} = C_8,\\
&\left\|  (k_j)_{xx} \right\|_{ L^1 (\Omega)} = \frac{1}{d_j} \left\|  k_j \right\|_{ L^1 (\Omega)} = \frac{1}{d_j},\\
&\left\| (k_j)_{xxx} \right\|_{ L^1 (\Omega)} = \frac{1}{d_j} \left\| (k_j)_x \right\|_{ L^1 (\Omega)} = \frac{C_8}{d_j},
\end{align*}
and
\begin{equation}\label{kj:C}
\left\|  (k_j)_x \right\|_{ C (\Omega)} = \frac{1}{2 d_j}.	
\end{equation}
\end{lemma}
As the proof is elementary, we put it in  \ref{append:kjL1}.

\subsection{Boundedness of auxiliary factors}
Next, we estimate the boundedness of the solutions $(\rho^\eps,\{v_j^\eps\}_{j=1}^M)$ to \eqref{eq:system} in this subsection.
First, we obtain the following lemma.
\begin{lemma}\label{lemm:k_j*rho_t}
Let $\rho^\eps$ be the solution to the first equation of \eqref{eq:system}. 
There exist positive constants $C_9$ and $C_{10}$ depending on $j$ but independent of $\eps$ such that
\begin{align*}
&\left\| k_j * \rho_t^\eps \right\|_{ C([0,T],L^2 (\Omega))} 
\le C_9,\\
&\left\| (k_j)_x * \rho_t^\eps \right\|_{ C([0,T],L^2 (\Omega))}
\le C_{10}, \\
&\left\| (k_j)_{xx} * \rho_t^\eps \right\|_{C([0,T], L^2 (\Omega))}
\le \frac{C_9}{d_j}.
\end{align*}
\end{lemma}
Although we should explicitly write the dependence of $t$ on the norm of the spatial direction for functions depending on the position $x$ and time $t$, for example $\|\cdot\|_{L^2(\Omega)}(t)$, we abbreviate the symbol of $(t)$ for the simple descriptions from here.
\begin{proof}[Proof of Lemma \ref{lemm:k_j*rho_t}]
From the first equation in \eqref{eq:system} and the final assertion of Corollary \ref{cor:mildsol}, we have 
\begin{align*}
k_j * \rho^\eps_t 
&= k_j * \rho^{\eps}_{xx} +  k_j * \frac{ \partial  }{ \partial x } \Big(  \rho^{\eps}  \frac{ \partial  }{ \partial x } \sum_{j=1}^M a_j v_j^\eps \Big)\\
&= (k_j)_x * \rho^{\eps}_{x} +  (k_j)_x * \Big(  \rho^{\eps} \sum_{j=1}^M a_j (v_{j}^\eps)_x \Big).
\end{align*}
We can define $k_j * \rho^\eps_t $ based on the final assertion of Corollary \ref{cor:mildsol}.
Then we compute that 
\begin{align*}
\left\| k_j * \rho_t^\eps \right\|_{ L^2 (\Omega)} 
&\le \left\|  (k_j)_x * \rho^{\eps}_{x}  \right\|_{ L^2 (\Omega)} + \left\| (k_j)_x * \Big(  \rho^{\eps}  \sum_{j=1}^M a_j (v_{j}^\eps)_x \Big) \right\|_{ L^2 (\Omega)} \\
& \le \left\|  (k_j)_x  \right\|_{ L^1 (\Omega)} \left\| \rho^{\eps}_{x} \right\|_{ L^2 (\Omega)} + \left\| (k_j)_x \right\|_{ L^1 (\Omega)} \left\|  \rho^{\eps}  \sum_{j=1}^M a_j (v_{j}^\eps)_x \right\|_{ L^2 (\Omega)} \\
&\le C_8 \left\| \rho^{\eps}_{x} \right\|_{ L^2 (\Omega)} + \sqrt{2L} C_8 \left\| \rho^{\eps } \right\|_{ C (\Omega)}  \sum_{j=1}^M |a_j| \left\| (v_{j}^\eps)_x \right\|_{ C (\Omega)} \\
&\le \sup_{t\in [0,T]} \Big\{ C_8 \left\| \rho^{\eps}_{x} \right\|_{ L^2 (\Omega)} + \sqrt{2L} C_8 \left\| \rho^{\eps } \right\|_{ C (\Omega)}  \sum_{j=1}^M |a_j|  
 \Big( \left\|  ( v_{j})_{0,x}  \right\|_{ C (\Omega)}	+ C_3 \left\| \rho^\eps \right\|_{  L^2 (\Omega)  }  \Big)\Big\} \\
&=: C_9, 
\end{align*}
where we used the Minkowski inequality, the Young inequality, \eqref{boun:v_j,x} in Lemma \ref{lem:bound_vj},  Lemma \ref{lemm:k_j_L^1}, and  $\rho^\eps \in C(\Omega) $ from the Sobolev's embedding theorem owing to $\rho^\eps \in C([0,T], H^1(\Omega) )$.
Similarly, we can obtain 
\begin{align*}
(k_j)_x * \rho^\eps_t 
= (k_j)_{xx} * \rho^{\eps}_{x} +  (k_j)_{xx} * \Big(  \rho^{\eps} \sum_{j=1}^M a_j (v_{j}^\eps)_x \Big)
\end{align*}
from the first equation in \eqref{eq:system}.
It yields that
\begin{align*}
\left\| (k_j)_x * \rho_t^\eps \right\|_{ L^2 (\Omega)}
&\le \left\|  (k_j)_{xx} * \rho^{\eps}_{x}  \right\|_{ L^2 (\Omega)} + \left\| (k_j)_{xx} * \Big(  \rho^{\eps}  \sum_{j=1}^M a_j (v_{j}^\eps)_x \Big) \right\|_{ L^2 (\Omega)} \\
& \le \left\|  (k_j)_{xx}  \right\|_{ L^1 (\Omega)} \left\| \rho^{\eps}_{x} \right\|_{ L^2 (\Omega)} + \left\| (k_j)_{xx} \right\|_{ L^1 (\Omega)} \left\|  \rho^{\eps}  \sum_{j=1}^M a_j (v_{j}^\eps)_x \right\|_{ L^2 (\Omega)} \\
&\le \frac{1}{d_j} \left\| \rho^{\eps}_{x} \right\|_{ L^2 (\Omega)} + \frac{\sqrt{2L}}{d_j} \left\| \rho^{\eps } \right\|_{ C (\Omega)}  \sum_{j=1}^M |a_j| \left\| (v_{j}^\eps)_x  \right\|_{ C (\Omega)} \\
&\le \sup_{t\in [0,T]} \Big\{ \frac{1}{d_j} \left\| \rho^{\eps}_{x} \right\|_{ L^2 (\Omega)}  + \frac{\sqrt{2L}}{d_j} \left\| \rho^{\eps } \right\|_{ C (\Omega)}  \sum_{j=1}^M |a_j|
\Big(  \left\|  ( v_{j})_{0,x}  \right\|_{ C (\Omega)} + C_3 \left\| \rho^\eps \right\|_{  L^2 (\Omega) } \Big)\Big\} \\
&=: C_{10}.
\end{align*}
Because of  the linearity of the equation, using the same calculation, we compute that 
\begin{align*}
\left\| (k_j)_{xx} * \rho_t^\eps \right\|_{ L^2 (\Omega)}
&\le \left\|  (k_{j})_{xxx} * \rho^{\eps}_{x}  \right\|_{ L^2 (\Omega)} + \left\| (k_{j})_{xxx} * \Big(  \rho^{\eps}  \sum_{j=1}^M a_j (v_{j}^\eps)_x \Big) \right\|_{ L^2 (\Omega)} \\
& \le \left\|  (k_{j})_{xxx}  \right\|_{ L^1 (\Omega)} \left\| \rho^{\eps}_{x} \right\|_{ L^2 (\Omega)} + \left\| (k_{j})_{xxx} \right\|_{ L^1 (\Omega)} \left\|  \rho^{\eps}  \sum_{j=1}^M a_j (v_{j}^\eps)_x \right\|_{ L^2 (\Omega)} \\
&\le \Big( C_8 \left\| \rho^{\eps}_{x} \right\|_{ L^2 (\Omega)} + \sqrt{2L}C_8 \left\| \rho^{\eps } \right\|_{ C (\Omega)}  \sum_{j=1}^M |a_j| \left\| (v_{j}^\eps)_x  \right\|_{ C (\Omega)}\Big)\frac{1}{d_j} \\
&\le  \frac{C_9}{d_j}.
\end{align*}
\end{proof}

Now we estimate the difference between the solutions in the following auxiliary equations
\begin{align}
	\eps (v_{j}^\eps)_t& = d_j (v_{j}^\eps)_{xx} - v_j^\eps + \rho^\eps,\label{eq:v_j^eps}\\ 
	0 & = d_j (v_{j})_{xx} - v_j + \rho^\eps \label{eq:v_j}
\end{align}
for $j= 1, \ldots, M$, where $\rho^\eps$ denotes the solution to \eqref{eq:system}.
We note that the solution to \eqref{eq:v_j} is given by $v_j = k_j * \rho^\eps$.
We set the difference as 
\begin{equation*}
	V_j^\eps(x,t) : = v_j^\eps(x,t) - v_j(x,t).
\end{equation*}
\begin{lemma}\label{lemm:V_j-L^2}
Let $C_9$ and $C_{10} $ be positive constants in Lemma \ref{lemm:k_j*rho_t}.
Assume the initial condition \eqref{conc_init_data:v_j}.
Then for any time $T>0$ the following estimates hold:
\begin{align*}
& \left\| V_j^\eps \right\|_{ C([0,T], L^2 (\Omega))}^2 + d_j \left\| (V_{j}^\eps)_x \right\|_{ L^2(0,T, L^2 (\Omega))}^2 \le C_9^2 \eps^2 \Big( 1 + \frac{ T}{2}  \Big),\\
& \left\| (V_{j}^\eps)_x \right\|_{C([0,T], L^2 (\Omega))}^2 + d_j \left\| (V_{j}^\eps)_{xx} \right\|_{ L^2(0,T, L^2 (\Omega))}^2 \le C_{10}^2 \eps^2 \Big( 1 + \frac{ T}{2} \Big), \\
& \left\| (V_{j}^\eps)_{xx} \right\|_{ C([0,T], L^2 (\Omega))}^2  + d_j \left\| (V_{j}^\eps)_{xxx} \right\|_{ L^2(0,T, L^2 (\Omega))}^2 \le \frac{C_9^2 \eps^2}{d_j^2} \Big( 1 + \frac{ T}{2} \Big).
\end{align*}
\end{lemma}
When excluding the terms multiplied by $d_j$ on the left-hand sides, the above inequality holds without $T/2$ on the right-hand sides.
\begin{proof}[Proof of Lemma \ref{lemm:V_j-L^2}]
Taking the difference between equations of \eqref{eq:v_j^eps} and \eqref{eq:v_j}, we see that 
\begin{align}
	\eps (V_j^\eps)_t 
	&= - \eps v_{j,t} + d_j (V_{j}^\eps)_{xx} - V_j^\eps\notag\\
	& = - \eps k_j * \rho^\eps_t + d_j (V_{j}^\eps)_{xx} - V_j^\eps.\label{eq:V_j^eps}
\end{align}
Multiplying this equation by $V_j^\eps$, integrating over $\Omega$ and using Lemma \ref{lemm:k_j*rho_t}, we have
\begin{align}
\frac{\eps}{2} \frac{d}{dt} \left\| V_j^\eps \right\|_{ L^2 (\Omega)}^2
&= -\eps \int_\Omega (k_j * \rho^\eps_t V_j^\eps) - d_j \left\| (V_{j}^\eps)_x \right\|_{ L^2 (\Omega)}^2 - \left\| V_{j}^\eps \right\|_{ L^2 (\Omega)}^2 \notag\\
&\le \eps \left\| k_j * \rho^\eps_t \right\|_{ L^2 (\Omega)} \left\| V_j^\eps \right\|_{ L^2 (\Omega)} - d_j\left\| (V_{j}^\eps)_x \right\|_{ L^2 (\Omega)}^2 - \left\| V_{j}^\eps \right\|_{ L^2 (\Omega)}^2 \notag\\
&\le \frac{\eps^2}{2} \left\| k_j * \rho_t^\eps \right\|_{ L^2 (\Omega)}^2 - d_j\left\| (V_{j}^\eps)_x \right\|_{ L^2 (\Omega)}^2 - \frac{1}{2} \left\| V_{j}^\eps \right\|_{ L^2 (\Omega)}^2 \notag\\
&\le \frac{\eps^2}{2} C_9^2 - d_j\left\| (V_{j}^\eps)_x \right\|_{ L^2 (\Omega)}^2 - \frac{1}{2} \left\| V_{j}^\eps \right\|_{ L^2 (\Omega)}^2 \label{ineq:V_{j,x}}.
\end{align}
Applying to the classical Gronwall lemma to 
\begin{equation*}
	\frac{d}{dt} \left\| V_j^\eps \right\|_{ L^2 (\Omega)}^2 \le \eps C_9^2 - \frac{1}{\eps} \left\| V_{j}^\eps \right\|_{ L^2 (\Omega)}^2,
\end{equation*}
we have that 
\begin{equation}\label{ineq:Vj}
	\left\| V_j^\eps \right\|_{ L^2 (\Omega)}^2 \le \left\| V_j^\eps(\cdot, 0) \right\|_{ L^2 (\Omega)}^2 e^{-t/\eps } + \eps^2 C_9^2 ( 1 - e^{-t/\eps } )
	\le  \eps^2 C_9^2
\end{equation}
from the initial conditions given in \eqref{conc_init_data:v_j}.
Furthermore, integrating \eqref{ineq:V_{j,x}} over $(0,T)$, we see that 
\begin{align}\label{ineq:L2L2_Vjx}
	 d_j \left\| (V_{j}^\eps)_x \right\|_{ L^2(0,T, L^2 (\Omega))}^2
	  \le \frac{\eps}{2} \left\| V_j^\eps \right\|_{ L^2 (\Omega)}^2 + d_j \left\| (V_{j}^\eps)_x \right\|_{ L^2(0,T, L^2 (\Omega))}^2 \le \frac{\eps^2 C_9^2 T}{2}.
\end{align}
Applying $\sup_{t \in [0,T]}$ to \eqref{ineq:Vj} and adding it to \eqref{ineq:L2L2_Vjx}, we obtain the first assertion.

Similarly, 
multiplying \eqref{eq:V_j^eps} by $-(V_{j}^\eps)_{xx}$ and integrating over $\Omega$, we have
\begin{align}
	\frac{\eps}{2} \frac{d}{dt} \left\| (V_{j}^\eps)_x \right\|_{ L^2 (\Omega)} ^2
	& = -\eps \int_\Omega ((k_j)_x * \rho^\eps_t (V_{j}^\eps)_x) - d_j \left\| (V_{j}^\eps)_{xx} \right\|_{ L^2 (\Omega)}^2 - \left\| (V_{j}^\eps)_x \right\|_{ L^2 (\Omega)}^2\notag \\
	& \le \frac{\eps^2}{2} C_{10}^2 - d_j\left\| (V_{j}^\eps)_{xx} \right\|_{ L^2 (\Omega)}^2 - \frac{1}{2} \left\| (V_{j}^\eps)_x \right\|_{ L^2 (\Omega)}^2, \label{ineq:V_{j,xx}}
\end{align}
where we used Lemma \ref{lemm:k_j*rho_t}.
From the Gronwall inequality, we have 
\begin{equation}\label{ineq:Vjx}
	\left\| (V_{j}^\eps)_x \right\|_{ L^2 (\Omega)}^2 \le \left\| (V_{j}^\eps)_x(\cdot, 0) \right\|_{ L^2 (\Omega)}^2 e^{-t/\eps } + \eps^2 C_{10}^2 ( 1 - e^{-t/\eps } )
	\le  \eps^2 C_{10}^2
\end{equation}
by the initial condition \eqref{conc_init_data:v_j}.
Integrating \eqref{ineq:V_{j,xx}} over $(0,T)$, we see that 
\begin{align}\label{ineq:L2L2_Vjxx}
d_j \left\| (V_{j}^\eps)_{xx} \right\|_{ L^2(0,T, L^2 (\Omega))}^2\le
	\frac{\eps}{2} \left\| (V_{j}^\eps)_x (\cdot, T)\right\|_{ L^2 (\Omega)}^2 + d_j \left\| (V_{j}^\eps)_{xx} \right\|_{ L^2(0,T, L^2 (\Omega))}^2 \le \frac{\eps^2 C_{10}^2 T}{2}.
\end{align}
Applying $\sup_{t \in [0,T]}$ to \eqref{ineq:Vjx} and adding it to \eqref{ineq:L2L2_Vjxx}, we have the second assertion.

Because of the linearity and $v_j^\eps \in L^2(0,T,H^3(\Omega))$, the same calculation can be applied.
We see that 
\begin{equation}\label{ineq:V_{j,xxx}}
		\frac{\eps}{2} \frac{d}{dt} \left\| (V_{j}^\eps)_{xx} \right\|_{ L^2 (\Omega)} ^2
		\le \frac{\eps^2}{2} \frac{C_{9}^2}{d_j^2} - d_j\left\| (V_{j}^\eps)_{xxx} \right\|_{ L^2 (\Omega)}^2 - \frac{1}{2} \left\| (V_{j}^\eps)_{xx} \right\|_{ L^2 (\Omega)}^2.
\end{equation}
Thus, the Gronwall lemma yields that 
\begin{equation}\label{ineq:Vjxx}
	\left\| (V_{j}^\eps)_{xx} \right\|_{ L^2 (\Omega)}^2 (t)\le \left\| (V_{j}^\eps)_{xx}(\cdot, 0) \right\|_{ L^2 (\Omega)}^2 e^{-t/\eps } +  \frac{\eps^2C_9^2}{d_j^2}  ( 1 - e^{-t/\eps } )
	\le  \frac{\eps^2C_9^2}{d_j^2} 
\end{equation}
from the initial condition given in \eqref{conc_init_data:v_j}.
Integrating \eqref{ineq:V_{j,xxx}} over $(0,T)$ with respect to $t$ and \eqref{ineq:Vjxx} implies the final assertion of this lemma.
\end{proof}

\subsection{Order estimation}\label{sec:order}
Under the above preparation, we estimate the difference of solutions.
Set the difference between the solutions to the first component of \eqref{eq:system} and \eqref{nonlocal:FP} as 
\begin{equation*}
	U^\eps (x,t) : = \rho^\eps(x,t) - \rho(x,t).
\end{equation*}
We will show the following convergence.
\begin{lemma}\label{lemma:L2_conv}
Suppose that $M$ is an arbitrarily fixed natural number.
Let $\rho$ be the solution to \eqref{nonlocal:FP} equipped with $W= \sum_{j=1}^M a_j k_j$ and the initial value $\rho_0 \in C^2(\Omega)$,  and let  $\rho^\eps$ be the solution to the first component of \eqref{eq:system} with \eqref{iv:rho} and \eqref{conc_init_data:v_j}.
 Then, for any $\eps>0$ and $T>0$, there exists a positive constant $C_{11}$ that depends on $a_j$ and $T$,  but is independent of $\eps$ such that 
\begin{equation}\label{L2:conv}
	 \left\| U^\eps \right\|_{C([0, T], L^2 (\Omega))}^2
	+\left\| U^\eps_x \right\|_{ L^2( 0,T, L^2 (\Omega))}^2  
	\le C_{11} \eps^2.
\end{equation}
\end{lemma}
\begin{proof}[Proof of Lemma \ref{lemma:L2_conv}]
Taking the difference between the first equation of \eqref{eq:system} and the equation of \eqref{nonlocal:FP}, we have 
\begin{equation}\label{eq:Ueps}
	U^\eps_t = U^\eps_{xx} - \sum_{j=1}^M a_j \frac{\partial}{\partial x} \Big( \rho^\eps  (V_{j}^\eps)_x  +  U^\eps (k_j)_x *\rho^\eps  + \rho (k_j)_x*U^\eps \Big).
\end{equation}
Subsequently, multiplying \eqref{eq:Ueps} by $U^\eps$ and integrating over $\Omega$, we obtain 
\begin{align*}
	\frac{1}{2} \frac{d}{dt} \left\| U^\eps \right\|_{ L^2 (\Omega)}^2 (t)
	&= -\left\| U^\eps_x \right\|_{ L^2 (\Omega)}^2 (t) +  \sum_{j=1}^M a_j \int_\Omega U^\eps_x \Big( \rho^\eps (V_{j}^\eps)_x  +  U^\eps (k_j)_x *\rho^\eps  + \rho  (k_j)_x*U^\eps\Big)(x,t) dx\\
	& = -\left\| U^\eps_x \right\|_{ L^2 (\Omega)}^2 (t)+I_1(t) +I_2(t) +I_3(t),
\end{align*}
where each term of the integral is set as 
\begin{align*}
	&I_1(t):=\sum_{j=1}^M a_j  \int_\Omega (U^\eps_x \rho^\eps  (V_{j}^\eps)_x) (x,t) dx, \\
 	&I_2(t):= \sum_{j=1}^M a_j  \int_\Omega ( U^\eps_x U^\eps  (k_j)_x *\rho^\eps ) (x,t) dx, \\
	&I_3(t):=\sum_{j=1}^M a_j  \int_\Omega ( U^\eps_x\rho  (k_j)_x*U^\eps ) (x,t) dx,
\end{align*}
respectively.
First, we compute $I_1$.
Using the Cauchy--Schwartz inequality, Sobolev embedding theorem, and Lemma \ref{lemm:V_j-L^2}, we have
\begin{align*}
	I_1
	&\le \sum_{j=1}^M |a_j|  \int_\Omega | U^\eps_x| | \rho^\eps | | (V_{j}^\eps)_x | dx \\
	& \le \left\| \rho^\eps \right\|_{ C (\Omega)} \sum_{j=1}^M |a_j|   \left\| U^\eps_x\right\|_{ L^2 (\Omega)}  \left\| (V_{j}^\eps)_x \right\|_{ L^2 (\Omega)} \\
	& \le \eps \left\| \rho^\eps \right\|_{ C (\Omega)}  \sum_{j=1}^M ( C_{10} |a_j| )\left\| U^\eps_x\right\|_{ L^2 (\Omega)} \\
	& \le  \frac 1 2 \left\| U^\eps_x\right\|_{ L^2 (\Omega)}^2 + \frac{\eps^2 C_{12} }{2} ,
\end{align*}
where 
\[
	C_{12} := \left\| \rho^\eps \right\|_{ C([0,T],C (\Omega))}^2 \Big( \sum_{j=1}^M ( C_{10} |a_j| ) \Big)^2.
\]
Next, we compute that 
\begin{align*}
	I_2 
	& \le \sum_{j=1}^M | a_j | \left\| \rho^\eps \right\|_{ C (\Omega)} \left\| (k_j)_x \right\|_{ L^1 (\Omega)} \left\| U_x^\eps \right\|_{ L^2 (\Omega)} \left\| U^\eps \right\|_{ L^2 (\Omega)} \\
	& \le  \left\| \rho^\eps \right\|_{ C (\Omega)}  \sum_{j=1}^M ( C_8| a_j |) \left\| U_x^\eps \right\|_{ L^2 (\Omega)} \left\| U^\eps \right\|_{ L^2 (\Omega)} \\
	& \le \frac{1}{4} \left\| U_x^\eps \right\|_{ L^2 (\Omega)}^2 + C_{13}^2  \left\| U^\eps \right\|_{ L^2 (\Omega)}^2,
\end{align*}
where we used the estimate in Lemma \ref{lemm:k_j_L^1} and  
\begin{equation*}
	C_{13} := \left\| \rho^\eps \right\|_{ C ([0,T],C(\Omega))}\sum_{j=1}^M (  C_8 | a_j | ) . 
\end{equation*}
Finally, the Sobolev embedding theorem, the Young inequality and the boundedness in Lemma \ref{lemm:k_j_L^1} yield that
\begin{align*}
	I_3 
	& \le \sum_{j=1}^M | a_j |  \left\| \rho \right\|_{ C (\Omega)} \int_\Omega | U^\eps_x| |  (k_j)_x*U^\eps | dx \\
	& \le  \left\| \rho \right\|_{ C (\Omega)}  \sum_{j=1}^M( C_8 | a_j | ) \left\| U_x^\eps \right\|_{ L^2 (\Omega)} \left\| U^\eps\right\|_{ L^2 (\Omega)} \\
	&\le \frac{1}{8}  \left\| U_x^\eps \right\|_{ L^2 (\Omega)}^2 + 2C_{14}^2  \left\| U^\eps\right\|_{ L^2 (\Omega)}^2,
\end{align*}
where  the constant is defined as  
\begin{equation*}
	C_{14} := \left\| \rho \right\|_{ C([0,T], C(\Omega))} \sum_{j=1}^M (  C_8 | a_j | ) .
\end{equation*}
Summarizing these estimations, we have 
\begin{align}
	\frac{1}{2} \frac{d}{dt} \left\| U^\eps \right\|_{ L^2 (\Omega)}^2 
	&\le - \frac{1}{8} \left\| U^\eps_x \right\|_{ L^2 (\Omega)}^2 + (C_{13}^2 + 2C_{14}^2 ) \left\| U^\eps \right\|_{ L^2 (\Omega)}^2 + \frac{\eps^2}{2} C_{12}\notag\\
	&= - \frac{1}{8} \left\| U^\eps_x \right\|_{ L^2 (\Omega)}^2 + \frac{C_{15}}{2} \left\| U^\eps \right\|_{ L^2 (\Omega)}^2 + \frac{\eps^2}{2} C_{12}, \label{ineq:L^2 U_x^eps}
\end{align}
where we put $C_{15}:= 2(C_{13}^2 + 2C_{14}^2 ) $.
Applying the classical Gronwall inequality with the initial condition \eqref{iv:rho} to 
\begin{equation*}
	\frac{d}{dt} \left\| U^\eps \right\|_{ L^2 (\Omega)}^2 \le C_{15} \left\| U^\eps \right\|_{ L^2 (\Omega)}^2 + \eps^2 C_{12},
\end{equation*}
we have 
\begin{equation}\label{ineq:U}
	\left\| U^\eps \right\|_{ L^2 (\Omega)}^2 (t)\le \frac{C_{12}}{C_{15}} (e^{ C_{15} t }  - 1) \eps^2.
\end{equation}
Furthermore, integrating \eqref{ineq:L^2 U_x^eps} over $(0,T)$, we also obtain that 
\begin{align}
	\frac{1}{4} \left\| U^\eps_x \right\|_{ L^2( 0,T, L^2(\Omega))}^2
	&\le\frac{1}{4} \left\| U^\eps_x \right\|_{ L^2( 0,T, L^2(\Omega))}^2  +\left\| U^\eps(\cdot, T) \right\|_{ L^2 (\Omega)}^2 \notag\\
	&\le  C_{15} \int^T_0 \left\| U^\eps (\cdot, t)\right\|_{ L^2 (\Omega)}^2 dt + C_{12} T \eps^2 \notag\\
	&\le  \eps^2 C_{12} e^{C_{15} T} T. \label{ineq:L2L2Ux}
\end{align}
Defining $C_{11} := C_{12} (e^{ C_{15} T }  - 1)/ C_{15} + 4  C_{12} e^{C_{15} T} T$ and adding \eqref{ineq:U} and \eqref{ineq:L2L2Ux} imply the assertion of this lemma.
\end{proof}
Similarly to this lemma, we can obtain the following convergence.
\begin{lemma}\label{lemma:H1_conv}
Suppose the same assumptions as Lemma \ref{lemma:L2_conv}.
 Then, for any $\eps>0$ and $T>0$, there exists a positive constant $C_{16}$ that depends on $a_j$ and $T$, and is independent of $\eps$
 such that 
\begin{equation*}
	\left\| U_x^\eps \right\|_{ C([0, T], L^2 (\Omega))}^2
	+\left\| U^\eps_{xx} \right\|_{ L^2( 0,T, L^2(\Omega))}^2  
	\le C_{16} \eps^2.
\end{equation*}
\end{lemma}
\begin{proof}[Proof of Lemma \ref{lemma:H1_conv}]
Similarly to Lemma \ref{lemma:L2_conv}, multiplying \eqref{eq:Ueps} by $-U^\eps_{xx}$ and integrating over $\Omega$, we have 
\begin{align*}
	\frac{1}{2} \frac{d}{dt} \left\| U^\eps_x \right\|_{ L^2 (\Omega)}^2 (t)
	&= -\left\| U^\eps_{xx} \right\|_{ L^2 (\Omega)}^2 (t)+ \sum_{j=1}^M a_j  \int_\Omega U^\eps_{xx} \frac{\partial}{\partial x} \Big( \rho^\eps   (V_{j}^\eps)_x  +  U^\eps  (k_j)_x *\rho^\eps  + \rho (k_j)_x*U^\eps \Big)(x,t)dx\\
	&= -\left\| U^\eps_{xx} \right\|_{ L^2 (\Omega)}^2 (t)
	+ \sum_{j=1}^M a_j  \int_\Omega U^\eps_{xx} \frac{\partial}{\partial x} ( \rho^\eps  (V_{j}^\eps)_x ) (x,t)dx \\
	&\quad + \sum_{j=1}^M a_j  \int_\Omega U^\eps_{xx} \frac{\partial}{\partial x} ( U^\eps  (k_j)_x *\rho^\eps ) (x,t)dx 
	+ \sum_{j=1}^M a_j  \int_\Omega U^\eps_{xx} \frac{\partial}{\partial x} ( \rho  (k_j)_x*U^\eps ) (x,t)dx \\
	& =: -\left\| U^\eps_{xx} \right\|_{ L^2 (\Omega)}^2 (t)+\mathcal{I}_1(t) +\mathcal{I}_2(t) +\mathcal{I}_3(t) +\mathcal{I}_4(t) +\mathcal{I}_5(t) +\mathcal{I}_6(t),
\end{align*}
where we defined each term of energy term by the integral as
\begin{equation*}
\begin{array}{ll}
\ds\mathcal{I}_1(t):=  \sum_{j=1}^M a_j  \int_\Omega ( U^\eps_{xx}   \rho^\eps_x  (V_{j}^\eps)_x ) (x,t) dx,
&\ds \mathcal{I}_2(t):= \sum_{j=1}^M a_j  \int_\Omega ( U^\eps_{xx}   \rho^\eps  (V_{j}^\eps)_{xx} ) (x,t) dx,\\
\ds\mathcal{I}_3(t):= \sum_{j=1}^M a_j  \int_\Omega ( U^\eps_{xx}  U^\eps_x  (k_j)_x *\rho^\eps ) (x,t) dx,
 &\ds\mathcal{I}_4(t):= \sum_{j=1}^M a_j  \int_\Omega ( U^\eps_{xx}  U^\eps  (k_j)_{xx} *\rho^\eps ) (x,t) dx,\\
\ds\mathcal{I}_5(t):= \sum_{j=1}^M a_j  \int_\Omega ( U^\eps_{xx}  \rho_x  (k_j)_x*U^\eps ) (x,t) dx,
&\ds\mathcal{I}_6(t):= \sum_{j=1}^M a_j  \int_\Omega ( U^\eps_{xx}  \rho  (k_j)_{xx}*U^\eps ) (x,t) dx.\\
\end{array}	
\end{equation*}

First, we estimate $\mathcal{I}_1$.
From Lemma \ref{lemm:V_j-L^2} and the Sobolev embedding theorem there exists a positive constant $C_s$ such that $\left\| V_{j,x} \right\|_{ C (\Omega)} \le C_s \left\| V_{j,x}  \right\|_{ H^1 (\Omega)} $.
Then we see that 
\begin{align*}
	\ds\mathcal{I}_1 
	&\le  \sum_{j=1}^M \left\| (V_{j}^\eps)_x \right\|_{ C (\Omega)} |a_j| \left\| U^\eps_{xx} \right\|_{ L^2 (\Omega)} \left\|  \rho^\eps_x  \right\|_{ L^2 (\Omega)} \\
	&\le \eps \sum_{j=1}^M C_s (C_{10} + \frac{C_9}{d_j}  )  |a_j| \left\|  \rho^\eps_x  \right\|_{ L^2 (\Omega)} \left\| U^\eps_{xx} \right\|_{ L^2 (\Omega)}  \\
	& \le \eps^2 C_{17} +  \frac{1}{8}  \left\| U^\eps_{xx} \right\|_{ L^2 (\Omega)} ^2,
\end{align*}
where 
\begin{equation*}
	C_{17} := 2\Big(  \left\|  \rho^\eps_x  \right\|_{ C([0,T], L^2 (\Omega))}  \sum_{j=1}^M C_s (C_{10} + \frac{C_9}{d_j}  )  |a_j| \Big)^2
\end{equation*}
Next, we compute $ \mathcal{I}_2$ as 
\begin{align*}
	 \mathcal{I}_2 
	 &\le  \left\| \rho^\eps  \right\|_{ C (\Omega)}  \sum_{j=1}^M |a_j|  \left\| U^\eps_{xx} \right\|_{ L^2 (\Omega)}  \left\|  (V_{j}^\eps)_{xx} \right\|_{ L^2 (\Omega)} \\
	 & \le \eps  \left\| \rho^\eps  \right\|_{ C (\Omega)}  \sum_{j=1}^M \Big( \frac{C_9}{d_j} |a_j| \Big) \left\| U^\eps_{xx} \right\|_{ L^2 (\Omega)}\\
	 & \le  \eps^2 C_{18} + \frac{1}{8} \left\| U^\eps_{xx} \right\|_{ L^2 (\Omega)}^2,
\end{align*}
where 
\begin{equation*}
	C_{18} := 2 \Big( \left\| \rho^\eps  \right\|_{ C([0,T], C(\Omega))}  \sum_{j=1}^M \Big( \frac{C_9}{d_j} |a_j| \Big)\Big)^2.
\end{equation*}
From $\rho^\eps \in C(\Omega)$ and  Lemma \ref{lemm:k_j_L^1} we see that 
\begin{align*}
	 \mathcal{I}_3 
	 &\le \left\| \rho^\eps \right\|_{ C (\Omega)}  \sum_{j=1}^M |a_j| \left\| (k_j)_x \right\|_{ L^1 (\Omega)} \left\| U^\eps_{xx}   \right\|_{ L^2 (\Omega)}  \left\| U^\eps_x \right\|_{ L^2 (\Omega)} \\
	 & \le \left\| \rho^\eps \right\|_{ C (\Omega)}  \sum_{j=1}^M ( C_8 |a_j| ) \left\| U^\eps_x \right\|_{ L^2 (\Omega)} \left\| U^\eps_{xx}   \right\|_{ L^2 (\Omega)}   \\
	 & \le \frac{ C_{19} }{2} \left\| U^\eps_x \right\|_{ L^2 (\Omega)}^2 +\frac{1}{8} \left\| U^\eps_{xx} \right\|_{ L^2 (\Omega)},
\end{align*}
where we put
\begin{equation*}
	C_{19} := 4\Big(  \left\| \rho^\eps \right\|_{ C([0,T],C (\Omega))}  \sum_{j=1}^M ( C_8 |a_j| ) \Big)^2.
\end{equation*}
Similarly to that of $\mathcal{I}_3$, we obtain that
\begin{align*}
	 \mathcal{I}_4 
	 &\le \left\| \rho^\eps \right\|_{ C (\Omega)}  \sum_{j=1}^M |a_j| \left\| (k_j)_{xx} \right\|_{ L^1 (\Omega)}  \left\| U^\eps_{xx}   \right\|_{ L^2 (\Omega)}  \left\| U^\eps \right\|_{ L^2 (\Omega)}\\
	 & \le \left\| \rho^\eps \right\|_{ C (\Omega)}  \sum_{j=1}^M \frac{ |a_j| }{ d_j }  \left\| U^\eps \right\|_{ L^2 (\Omega)} \left\| U^\eps_{xx}   \right\|_{ L^2 (\Omega)}  \\
	 & \le 2 \Big( \left\| \rho^\eps \right\|_{ C (\Omega)}  \sum_{j=1}^M \frac{ |a_j| }{ d_j } \Big)^2  \left\| U^\eps \right\|_{ L^2 (\Omega)}^2 +  \frac 1 8 \left\| U^\eps_{xx}   \right\|_{ L^2 (\Omega)}^2 \\
	 & \le C_{20} \eps^2+  \frac 1 8 \left\| U^\eps_{xx}   \right\|_{ L^2 (\Omega)}^2,
\end{align*}
where we used the estimate \eqref{L2:conv} in Lemma \ref{lemma:L2_conv} and we put
\begin{equation*}
	C_{20}:= 2 \Big( \left\| \rho^\eps \right\|_{ C([0,T],C (\Omega))}  \sum_{j=1}^M \frac{ |a_j| }{ d_j } \Big)^2 C_{11}.
\end{equation*}
Hereafter, we often use the estimate \eqref{L2:conv} in Lemma \ref{lemma:L2_conv}.
We compute $\mathcal{I}_5$ as 
\begin{align*}
	\mathcal{I}_5 
	& \le   \sum_{j=1}^M |a_j| \left\| (k_j)_x \right\|_{ C (\Omega)} \left\| U^\eps \right\|_{ L^1 (\Omega)} \left\|  \rho_x \right\|_{ L^2 (\Omega)} \left\|  U^\eps_{xx}  \right\|_{ L^2 (\Omega)} \\
	& \le \left\| U^\eps \right\|_{ L^1 (\Omega)}  \sum_{j=1}^M \frac{|a_j|}{2d_j} \left\|  \rho_x \right\|_{ L^2 (\Omega)} \left\|  U^\eps_{xx}  \right\|_{ L^2 (\Omega)}\\
	& \le C_{21}  \eps^2  +  \frac{1}{8} \left\|  U^\eps_{xx}  \right\|_{ L^2 (\Omega)}^2,
\end{align*}
where we used \eqref{kj:C} in Lemma \ref{lemm:k_j_L^1}  and \eqref{L2:conv} in Lemma \ref{lemma:L2_conv}, and  we put
\begin{equation*}
	C_{21} :=4L \Big(  \left\|  \rho_x \right\|_{ C([0,T], L^2 (\Omega))}  \sum_{j=1}^M  \frac{|a_j|}{2d_j}  \Big)^2 C_{11}.
\end{equation*}
Similarly, we see that
\begin{align*}
	\mathcal{I}_6 
	&\le \left\|  \rho \right\|_{ C (\Omega)}  \sum_{j=1}^M |a_j|  \left\|  U^\eps_{xx} \right\|_{ L^2 (\Omega)}  \left\|  (k_j)_{xx} \right\|_{ L^1 (\Omega)}  \left\| U^\eps  \right\|_{ L^2 (\Omega)} \\
	& \le  \left\|  \rho \right\|_{ C (\Omega)}  \sum_{j=1}^M \frac{ |a_j| }{ d_j } \left\| U^\eps  \right\|_{ L^2 (\Omega)}  \left\|  U^\eps_{xx} \right\|_{ L^2 (\Omega)}\\
	& \le 2\Big( \left\|  \rho \right\|_{ C (\Omega)}  \sum_{j=1}^M \frac{ |a_j| }{ d_j } \Big)^2  \left\| U^\eps  \right\|_{ L^2 (\Omega)}^2 + \frac{1}{8}  \left\|  U^\eps_{xx} \right\|_{ L^2 (\Omega)}\\
	&\le C_{22}\eps^2 +  \frac{1}{8}  \left\|  U^\eps_{xx} \right\|_{ L^2 (\Omega)},
\end{align*}
where 
\begin{equation*}
	C_{22} := 2\Big( \left\|  \rho \right\|_{ C([0,T], C (\Omega))}  \sum_{j=1}^M \frac{ |a_j| }{ d_j } \Big)^2 C_{11}.
\end{equation*}
Combining these estimation and setting a positive constant as
\begin{equation*}
	C_{23}:= 2( C_{17} + C_{18}+ C_{19} + C_{20} + C_{21} + C_{22} ),
\end{equation*}
we have 
\begin{equation}\label{ineq:L^2 U_xx^eps}
	\frac{1}{2} \frac{d}{dt} \left\| U^\eps_x \right\|_{ L^2 (\Omega)}^2 
	\le -\frac{1}{4}  \left\|  U^\eps_{xx} \right\|_{ L^2 (\Omega)}^2 + \frac{C_{19}}{2} \left\| U_x^\eps \right\|_{ L^2 (\Omega)}^2 +   \frac{C_{23} }{2}  \eps^2.
\end{equation}
Applying to Gronwall inequality to \eqref{ineq:L^2 U_xx^eps} without $-\left\|  U^\eps_{xx} \right\|_{ L^2 (\Omega)}^2/4$,
we obtain that 
\begin{equation}\label{ineq:Ux}
	\left\| U^\eps_x \right\|_{ L^2 (\Omega)}^2 (t) \le \left\| U^\eps_x (\cdot, 0) \right\|_{ L^2 (\Omega)}^2 +\frac{C_{23}}{C_{19}} ( e^{C_{19} t } -1)\eps^2
	= \frac{C_{23}}{C_{19}} ( e^{C_{19} t } -1)\eps^2
\end{equation}
by the initial condition in \eqref{iv:rho}.

Finally, integrating \eqref{ineq:L^2 U_xx^eps} over $(0,T)$, we obtain that 
\begin{equation}\label{ineq:L2L2Uxx}
	\frac 1 2  \left\|  U^\eps_{xx} \right\|_{ L^2 (0,T, L^2(\Omega))}^2 
	\le \left\| U^\eps_x (\cdot, T)\right\|_{ L^2 (\Omega)}^2 + \frac 1 2  \left\|  U^\eps_{xx} \right\|_{ L^2 (0,T, L^2(\Omega))}^2 
	\le \eps^2 C_{23} e^{C_{19} T} T.
\end{equation}
Adding \eqref{ineq:Ux} and \eqref{ineq:L2L2Uxx} yields the assertion of this lemma.
\end{proof}

Then we obtain the following proof.
\begin{proof}[Proof of Theorem \ref{thm:order_est}]
Putting 
\begin{align*}
    C_1:= 2\Big( C_{11} + C_{11}T + C_{16} \Big)^{1/2}, \quad 
    C_2:= C_1 + 2\Big( 1+\frac{T}{2} \Big)^{1/2} \Big( C_9^2 +C_9^2T+C_{10}^2\Big)^{1/2},
\end{align*}
where we used Lemma \ref{lemma:L2_conv}, Lemma \ref{lemma:H1_conv} for $C_1$ and Lemma \ref{lemm:V_j-L^2} for $C_2$.
We obtain the convergence of the assertion in this theorem.
\end{proof}

Next, we prove Lemma \ref{lemm:rho1_rho2}.
\begin{proof}[Proof of Lemma \ref{lemm:rho1_rho2} and Remark \ref{rem:rho1_rho2}]
Set the differences between the solutions and between the kernels as 
\begin{equation*}
	U(x,t):= \rho_1(x,t) - \rho_2(x,t), \quad W_{\mbox{e}} (x):= w_1(x)-w_2(x),
\end{equation*}
respectively.
The method for this proof is similar to that of Lemmas \ref{lemma:L2_conv} and \ref{lemma:H1_conv}.
Taking the difference between the equations $(\mbox{P}_1)$ and $(\mbox{P}_2)$,  we have
\begin{align}\label{eq:U}
	U_t = U_{xx} - \frac{\partial}{\partial x}( \rho_1w_{1,x} *U + \rho_1W_{\mbox{e},x}*\rho_2 +Uw_{2,x}*\rho_2 ).
\end{align}
Multiplying it by $U$ and integrating it over $\Omega$, we have 
\begin{align*}
	\frac{1}{2} \frac{d}{dt} \left\| U \right\|_{ L^2 (\Omega)} ^2
	= -\left\| U_x \right\|_{ L^2 (\Omega)} ^2 + \int_{\Omega} ( U_x(  \rho_1w_{1,x} *U + \rho_1W_{\mbox{e},x}*\rho_2 +Uw_{2,x}*\rho_2) ) .
\end{align*}
Since
\begin{align*}
	 \int_{\Omega} ( U_x \rho_1w_{1,x} *U ) 
	 &\le  \left\| \rho_1 \right\|_{ C (\Omega)}  \left\| w_1 \right\|_{ L^1 (\Omega)} \left\| U_x \right\|_{ L^2 (\Omega)}^2, \\
	 \int_{\Omega} ( U_x \rho_1W_{\mbox{e},x}*\rho_2 )  
	 &\le \left\| \rho_1 \right\|_{ C (\Omega)} \left\| W_{\mbox{e}} \right\|_{ L^1 (\Omega)}\left\| \rho_{2,x} \right\|_{ L^2 (\Omega)}  \left\| U_x \right\|_{ L^2 (\Omega)},\\
	  \int_{\Omega} ( U_xUw_{2,x}*\rho_2 
	  & \le \left\| U \right\|_{ C (\Omega)}  \left\| U_x \right\|_{ L^2 (\Omega)} \left\| w_2 \right\|_{ L^1 (\Omega)} \left\| \rho_{2,x} \right\|_{ L^2 (\Omega)},
\end{align*}
we can compute that 
\begin{align}\label{ineq:Ux_L2}
	\frac{1}{2} \frac{d}{dt} \left\| U \right\|_{ L^2 (\Omega)} ^2
	\le -\frac{1}{4} \left\| U_x \right\|_{ L^2 (\Omega)}^2 + \frac{C_{24}}{2}  \left\| U \right\|_{ L^2 (\Omega)} ^2 + \frac{C_{25}}{2} \left\| W_{\mbox{e}} \right\|_{ L^1 (\Omega)}^2
\end{align}
with suitable positive constants $C_{24}$ and $C_{25}$.
Applying the Gronwall inequality to this, we have
\begin{equation}\label{ineq:rho12}
	\left\| U \right\|_{ C([0,T], L^2 (\Omega))}^2 \le \frac{C_{25}}{C_{24}} (e^{C_{24} T} -1 )\left\| W_e \right\|_{ L^1 (\Omega)}^2.
\end{equation}
Integrating \eqref{ineq:Ux_L2} over $[0,T]$ and adding it and \eqref{ineq:rho12}, we have the assertion for this Lemma.

Next, multiplying \eqref{eq:U} by $-U_{xx}$ and integrating it over $\Omega$, we consider the following equation of energy
\begin{align*}
	\frac{1}{2} \frac{d}{dt} \left\| U_{x} \right\|_{ L^2 (\Omega)} ^2
	=-\left\| U_{xx}\right\|_{ L^2 (\Omega)} ^2 + \int_{\Omega} ( U_{xx}  \frac{\partial }{\partial x} (  \rho_1w_{1,x} *U + \rho_1W_{e,x}*\rho_2 +Uw_{2,x}*\rho_2) ).
\end{align*}
Since 
\begin{align*}
& \int_{\Omega} ( U_{xx} ( \rho_{1,x}w_{1,x} *U + \rho_1w_1 *U_{xx}) ) \\
 &\le \left\| w_{1,x} \right\|_{ L^1 (\Omega)} \left\| U \right\|_{ C (\Omega)} \left\| U_{xx} \right\|_{ L^2 (\Omega)} \left\| \rho_{1,x} \right\|_{ L^2 (\Omega)} 
 + \left\| \rho_1\right\|_{ C (\Omega)}  \left\| w_{1}\right\|_{ L^1 (\Omega)} \left\| U_{xx} \right\|_{ L^2 (\Omega)}^2,\\
 &\int_{\Omega} ( U_{xx} ( \rho_{1,x}W_{\mbox{e},x}*\rho_2 +  \rho_1W_{\mbox{e},x}*\rho_{2,x}) )  \\
 & \le \left\| \rho_2 \right\|_{ C (\Omega)} \left\| W_{\mbox{e},x} \right\|_{ L^1 (\Omega)} \left\| U_{xx} \right\|_{ L^2 (\Omega)} \left\| \rho_{1,x} \right\|_{ L^2 (\Omega)} 
 + \left\| \rho_1 \right\|_{ C (\Omega)} \left\| W_{\mbox{e},x} \right\|_{ L^1 (\Omega)} \left\| \rho_{2,x} \right\|_{ L^2 (\Omega)} \left\| U_{xx} \right\|_{ L^2 (\Omega)},\\
 &\int_{\Omega} ( U_{xx} ( U_xw_{2,x}*\rho_2 + U w_{2,x}*\rho_{2,x} ) )  \\
 &\le \left\| \rho_2 \right\|_{ C (\Omega)} \left\|w_{2,x} \right\|_{ L^1 (\Omega)}  \left\| U_x \right\|_{ L^2 (\Omega)}  \left\| U_{xx} \right\|_{ L^2 (\Omega)}
 + \left\| \rho_{2,x} \right\|_{ L^2 (\Omega)} \left\|w_{2,x} \right\|_{ L^1 (\Omega)} \left\| U \right\|_{ C (\Omega)}  \left\| U_{xx} \right\|_{ L^2 (\Omega)},
\end{align*}
we can obtain that
\begin{align}
	\frac{1}{2} \frac{d}{dt} \left\| U_{x} \right\|_{ L^2 (\Omega)} ^2
	&\le -\frac{1}{4} \left\| U_{xx} \right\|_{ L^2 (\Omega)}^2 + \frac{C_{26}}{2} \left\| U \right\|_{ L^2 (\Omega)}^2 +  \frac{C_{27}}{2} \left\| U_x \right\|_{ L^2 (\Omega)}^2
	+ \frac{C_{28}}{2}  \left\| W_{\mbox{e},x} \right\|_{ L^1 (\Omega)}^2 \notag\\
	&\le -\frac{1}{4} \left\| U_{xx} \right\|_{ L^2 (\Omega)}^2 +  \frac{C_{27}}{2} \left\| U_x \right\|_{ L^2 (\Omega)}^2 + \frac{C_{29}}{2}  \left\| W_{\mbox{e}} \right\|_{ W^{1,1} (\Omega)}^2 \label{ineq:Uxx_L2}
\end{align}
with suitable positive constants from $C_{26}$ to $C_{29}$.
Therefore, the Gronwall inequality yields that 
\begin{equation}\label{ineq:rho12x}
	\left\| U_x \right\|_{ C( [0,T], L^2 (\Omega))}^2 \le \frac{C_{29}}{C_{27}} (e^{C_{27} T} -1 )\left\| W_{\mbox{e}} \right\|_{ W^{1,1} (\Omega)}^2.
\end{equation}

Finally integrating \eqref{ineq:Ux_L2} and \eqref{ineq:Uxx_L2} over $(0,T)$ and adding them and \eqref{ineq:rho12} and \eqref{ineq:rho12x}, we obtain the  assertion for Remark \ref{rem:rho1_rho2}.
\end{proof}

\section{Coefficients of linear sum}\label{sec:cof_aj}
We now explain the method used for determining the coefficient of the linear sum of the fundamental solution $\{ a_j \}_{j=1}^M$ for a given even potential function $W$.
Furthermore, we will perform the numerical simulations of the approximation of $W$ by sum of $\cosh j(L-|x|)$, and numerical simulations of \eqref{nonlocal:FP} and \eqref{eq:system} with this series expansion.
Since $W$ is even, we only consider $[0,L]$.
First, we provide the following lemma with respect to the $n$ degree Chebyshev polynomial $T_n$.
We set coefficients as 
\begin{align*}
    C^n_k := \frac{n}{2} (-1)^k \frac{(n-k-1)!}{k!(n-2k)!} 2^{n-2k}, \quad (k=0, \cdots, \Big[\frac n 2 \Big] )
\end{align*}
for $n \in \N$, where $[\cdot]$ denotes the Gauss symbol.
By this constant the Chebyshev polynomial of $n$ degree $T_n$ can be expressed as $T_n(x) = \sum_{k=0}^{[n/2]}C^n_k x^{n-2k}$ for $x\in [-1,1]$.
Utilizing this equation, we have the following Lemma regarding the change of the variable for $T_n$.
\begin{lemma}\label{lemm:Cheby}
Setting  
\begin{equation*}
    \mu^n_{k,j}:= C^n_k \Big( \frac{2}{b-a} \Big)^{n-2k}\begin{pmatrix} n-2k \\ j \end{pmatrix} \Big( -\frac{b+a}{2} \Big)^j
\end{equation*}
for $n \in \N$, we define the coefficient as
\begin{equation*}
    \xi^n_k :=
    \left\{
    \begin{aligned}
        &\sum_{l=0}^{[n/2]-[(k+1)/2]}  \mu^n_{l,n-2l-k}, \quad \text{if $n$ is even}, \\ 
        &\sum_{l=0}^{[n/2]-[k/2]}  \mu^n_{l,n-2l-k}, \quad \text{otherwise}. \\ 
    \end{aligned}
    \right.
\end{equation*}
Then, 
\begin{equation*}
    T_n\Big( \frac{2x - (b+a)}{b-a} \Big) = \sum_{k=0}^n \xi^n_k x^k, \quad x \in [a,b]
\end{equation*}
holds.
\end{lemma}
\begin{proof}[Proof of Lemma \ref{lemm:Cheby}]
    We compute that
    \begin{align*}
        T_n\Big( \frac{2x - (b+a)}{b-a} \Big) 
        &= \sum_{k=0}^{[n/2]}C^n_k \Big( \frac{2x -(b+a) }{ b-a } \Big)^{n-2k}
        =\sum_{k=0}^{[n/2]}C^n_k \Big( \frac{2}{b-a} \Big)^{n-2k} \Big( x - \frac{b+a}{2}\Big)^{n-2k}\\
        & = \sum_{k=0}^{[n/2]}C^n_k \Big( \frac{2}{b-a} \Big)^{n-2k} \sum_{j=0}^{n-2k} \begin{pmatrix} n-2k \\ j \end{pmatrix} x^{n-2k-j} \Big( -\frac{b+a}{2} \Big)^j\\
        & = \sum_{k=0}^{[n/2]} \sum_{j=0}^{n-2k} \mu^n_{k,j} x^{n-2k-j} 
        = \sum_{k=0}^n \xi^n_k x^k,
    \end{align*}
where we used the binomial expansion in the third equality.
\end{proof}

Next, we explicitly provide the coefficient of the linear sum of the $n$ degree Lagrange interpolation polynomial with the Chebyshev nodes for the arbitrary function $F=F(x)$ for $x\in [a,b]$.
We will replace the arbitrary function $F$ with the function $f$ defined in \eqref{def:f} to prove Theorem \ref{thm:HCSE}.
The root of the $n$ degree Chebyshev polynomial, called Chebyshev nodes, in an arbitrary interval $[a,b]$ is given by
\begin{equation*}
    r^n_j:= \frac{a+b}{2} + \frac{b-a}{2}\cos \frac{2j+1}{2n}\pi, \quad (j=1,\ldots, n).
\end{equation*}
We have that
\begin{equation}\label{mpoly:Cheb}
\prod_{j=1}^n(x-r_j^n)=\frac{1}{2^n} \Big( \frac{b-a}{2} \Big)^{n+1} T_{n+1
}\Big( \frac{2x - (b+a)}{b-a} \Big), \ n \in \N.    
\end{equation}
Moreover, setting the coefficient as
\begin{equation}\label{cof:zeta}
    \zeta^n_j:= \frac{F(r^{n+1}_j)}{\prod_{k=0,k\neq j}^n(r^{n+1}_j - r^{n+1}_k)}, \quad (j=1,\ldots, n)
\end{equation}
for $n \in \N$,
we see that the $n$ degree Chebyshev polynomial for the function $F$ is given by
\begin{equation*}
    L_n(x):= \sum_{j=0}^n  \zeta^n_j \prod_{k=0,k\neq j}^n (x - r^{n+1}_k).
\end{equation*}
Then we obtain the following proposition. 
\begin{proposition}\label{prop:lag_coef}
    Set 
    \begin{align}
        &\beta_{l,j}^n := \sum_{k=l}^{n-1} (r^n_j)^{k-l} \xi^n_{k+1} \quad (j=1,\ldots,n),\notag\\
        &b^n_l:= \frac{1}{2^n} \Big( \frac{b-a}{2} \Big)^{n+1}  \sum_{j=0}^n \zeta^n_j \beta_{l,j}^{n+1} \quad (l=0,\ldots,n) \label{cof:b}
    \end{align}
    for $n \in \N$.
    Then the $n$ degree Lagrange interpolation polynomial $L_n$ for an arbitrary function $F$ om $[a,b]$ can be described as
    \begin{equation*}
        L_n(x)= \sum_{l=0}^n b^n_l x^l, \quad x \in [a,b].
    \end{equation*}
\end{proposition}
\begin{proof}[Proof of Proposition \ref{prop:lag_coef}]
    Using Lemma \ref{lemm:Cheby} and \eqref{mpoly:Cheb}, we can compute that 
    \begin{align*}
        L_n(x)
        &=\sum_{j=0}^n  \zeta^n_j \prod_{k=0,k\neq j}^n (x - r^{n+1}_k)\\
        &=\sum_{j=0}^n \zeta^n_j \frac{1}{2^n} \Big( \frac{b-a}{2} \Big)^{n+1} T_{n+1
}\Big( \frac{2x - (b+a)}{b-a} \Big) \frac{1}{x - r^{n+1}_j}\\
        &=\frac{1}{2^n} \Big( \frac{b-a}{2} \Big)^{n+1} \sum_{j=0}^n \zeta^n_j  \frac{ \sum_{k=0}^{n+1} \xi^{n+1}_k x^k}{x - r^{n+1}_j}. 
    \end{align*}
    As $x - r^{n+1}_j$ for $j = 0,\ldots,n$ are the factors of $T_n( (2x - (b+a) )/(b-a) )$, respectively, $\sum_{k=0}^{n+1} \xi^{n+1}_k x^k/(x - r^{n+1}_j )$ must be divisible from the factor theorem.
    Thus, we obtain that
    \begin{align*}
     L_n(x) 
     &=  \frac{1}{2^n} \Big( \frac{b-a}{2} \Big)^{n+1}  \sum_{j=0}^n\zeta^n_j \sum_{l=0}^{n} \beta^{n+1}_{l,j} x^l\\
     &=\sum_{l=0}^{n} \Big( \frac{1}{2^n} \Big( \frac{b-a}{2} \Big)^{n+1} \sum_{j=0}^n \zeta^n_j \beta^{n+1}_{l,j} \Big)x^l\\
     & =  \sum_{l=0}^n b^n_l x^l.
    \end{align*}
\end{proof}

Before the proof of Theorem \ref{thm:HCSE},  we introduce the following constant for $n \in \N$ and $j \equiv n\pmod 2$:
\begin{equation*}
    \delta^n_j:=
    \left\{
    \begin{aligned}
    &\frac{1}{2^{n-1} } \begin{pmatrix} n \\ \frac{n-j}{2} \end{pmatrix} \quad \text{if $j \neq 0$},\\
    &\frac{1}{2^{n} } \begin{pmatrix} n \\ \frac{n-j}{2} \end{pmatrix} \quad \text{if $j=0$}.\\
    \end{aligned}
    \right.
\end{equation*}
Using this constant,
we can describe the formula as $x^n = \sum_{j=0, j \equiv n\pmod 2}^n \delta^n_j T_j(x)$.
In addition, from the property of the Chebyshev polynomial, we note that $T_n(\cosh(L-|x|))= \cosh n(L-|x|)$ for $n \in \N$ holds.
Using $f$ in \eqref{def:f} rather than $F$ in \eqref{cof:zeta}, we reconsider the coefficients $\zeta^n_j$ in \eqref{cof:zeta} and $b^n_l$ in \eqref{cof:b} in Theorem \ref{thm:HCSE}.
\begin{theorem}\label{thm:HCSE}
    Assume that $W \in C^m([0,L])$ for given $2 \le m \in \N$ and $W$ is even.
    let $f$ be 
    \begin{equation}\label{def:f}
        f(x) : = W(L - \log(x+\sqrt{x^2-1})) = W( \cosh^{-1}(L-x)),
    \end{equation}
    and $\zeta^n_j$ and $b^n_l$ be \eqref{cof:zeta} with $f$ and \eqref{cof:b} with a natural number $n \le m-1$, respectively.
    Set the coefficient $\alpha^{n}_j $ as
    \begin{equation*}
        \alpha^n_j = \sum_{\substack{ k=j \\ k \equiv j \pmod 2}}^n b^n_k\delta^k_j.
    \end{equation*}
    Then, for any $n \le m-1$ 
    \begin{equation*}
         \Big\| W - \sum_{j=0}^{n} \alpha^{n}_j \cosh j(L-\cdot) \Big\|_{C([0,L])} 
        \le \frac{1}{2^{n}(n+1)!}  \Big(\frac{\cosh L-1}{2} \Big)^{n+1} \max_{y\in[1,\cosh L ]}|f^{(n+1)}(y) |
    \end{equation*}
    holds.
\end{theorem}
\begin{proof}[Proof of Theorem \ref{thm:HCSE}]
From the property of the Lagrange polynomial for $f(y)$ defined in \eqref{def:f} for any $y \in [1, \cosh L]$, we have
    \begin{equation*}
        \frac{1}{2^n(n+1)!}  \Big(\frac{\cosh L-1}{2} \Big)^{n+1} \max_{ y \in [1, \cosh L] }|f^{(n+1)}(y) |
        \ge \| f - L_n \|_{C([1,\cosh L])}.
    \end{equation*}
    Using Proposition \ref{prop:lag_coef} and changing the variable to $y = \cosh(L-x)$ for $x \in [0,L]$, we can compute the right-hand side of the above inequality as
    \begin{align*}
        \| f - L_n \|_{C([1,\cosh L])}
        &= \sup_{y \in [1,\cosh L]} | f(y)  - \sum_{j=0}^n b_j^n y^j|\\
        &=  \Big\| f  - \sum_{j=0}^n b_j^n \sum_{\substack{k=0 \\ k \equiv j\pmod 2}}^j \delta^j_k T_k\Big\|_{C([1,\cosh L])}\\
        &=\Big\| f(\cosh(L-\cdot))  - \sum_{j=0}^n b_j^n \sum_{\substack{k=0 \\ k \equiv j\pmod 2}}^j \delta^j_k T_k(\cosh(L-\cdot)) \Big\|_{C([0, L])}\\
        &=\Big\| W  - \sum_{j=0}^n b_j^n \sum_{\substack{k=0 \\ k \equiv j\pmod 2}}^j \delta^j_k \cosh k(L-\cdot) \Big\|_{C([0, L])}\\
         &=\Big| W  - \sum_{j=0}^n \sum_{\substack{k=j \\ k \equiv j\pmod 2}}^n b_k^n  \delta^k_j \cosh j(L-\cdot) \Big|_{C([0, L])}\\
         &=  \Big\| W - \sum_{j=0}^n \alpha^n_j \cosh j(L-\cdot) \Big\|_{C([0, L])}.
    \end{align*}
\end{proof}
Next, we prove Corollary \ref{cor:w-kj}.
\begin{proof}[Proof of Corollary  \ref{cor:w-kj}]
Using Theorem \ref{thm:HCSE} and $a_j$ defined in \ref{set:aj}, we can obtain the same estimate as the proof of \ref{thm:HCSE}.
\end{proof}

Then we explain the proof of Theorem \ref{thm:RDSA}.
\begin{proof}[Proof of Theorem \ref{thm:RDSA}]
According to  Theorem \ref{thm:HCSE}, for any even function $W$ in $C^\infty([0,L])$, and arbitrary $\eps'>0$ there exist a natural number $M>0$ and constants $\{ \alpha_j^{M-1}\}_{j=0}^{M-1}$ such that
\begin{equation*}
	\Big\| W - \sum_{j=0}^{M-1} \alpha_j^{M-1} \cosh j(L-|\cdot|) \Big\|_{C(\Omega)} < \eps'.
\end{equation*}
Here we set $n=M-1$ in Theorem \ref{thm:HCSE}.
Putting the parameters as \eqref{set:dj} and \eqref{set:aj}, and $\eps'=\eps$, then we have
\begin{equation*}
	\Big\| W - \sum_{j=1}^M a_j k_j\Big\|_{C(\Omega)} < \eps.
\end{equation*}
Let $\bar{u}(x,t)$ be the solution to \eqref{nonlocal:FP} with the integral kernel $\sum_{j=1}^M a_j k_j$.
Then Lemmas \ref{lemm:rho1_rho2} and \ref{lemma:L2_conv} yields that 
\begin{align*}
	\left\| u - u^\eps \right\|_{ C( [0,T], L^2 (\Omega))} 
	&\le \left\| u - \bar{u} \right\|_{ C( [0,T], L^2 (\Omega))} + \left\| \bar{u} - u^\eps \right\|_{ C( [0,T], L^2 (\Omega))}  \\
	&\le {\tilde C_T} \left\| W - \sum_{j=1}^M a_j k_j \right\|_{ L^1 (\Omega)}  + \sqrt{C_{11}} \eps\\
	&\le 2L {\tilde C_T} \eps +\sqrt{C_{11}} \eps \\
	& \le C_T \eps,
\end{align*}
where $C_T:=\max\{2L{\tilde C_T}, \sqrt{C_{11}} \}$.

\end{proof}
\begin{remark}
If $W$ is continuous rather than $C^\infty([0,L])$ under the assumption of Theorem \ref{thm:RDSA}, then the existence of $\{ a_j\}$ is guaranteed by Corollary 1 of Ninomiya et al. \cite{NTY2017}.
Thus, Theorem \ref{thm:RDSA} can be proven similarly.
\end{remark}

We performed a numerical simulation of the approximation for the potential $W$ by the linear combination of $\cosh j(L-|x|)$.
The results are shown in Figure \ref{fig:HCSE}.
\begin{figure}[bt]
\begin{center}
\begin{tabular}{ccc}
	\includegraphics[keepaspectratio, scale=0.4]{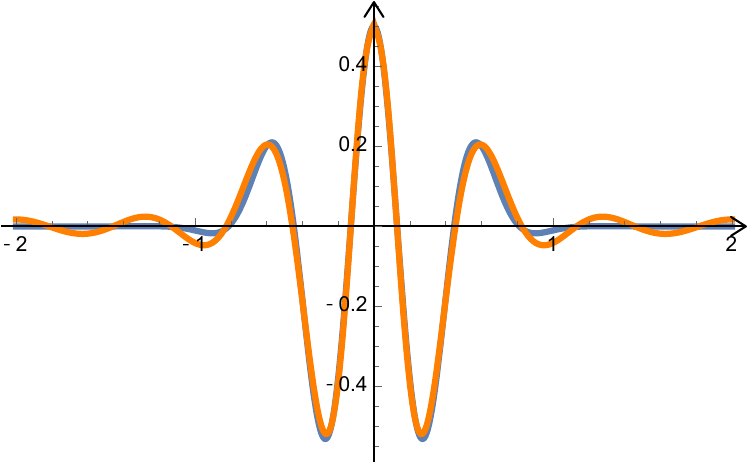}
	&\includegraphics[keepaspectratio, scale=0.35]{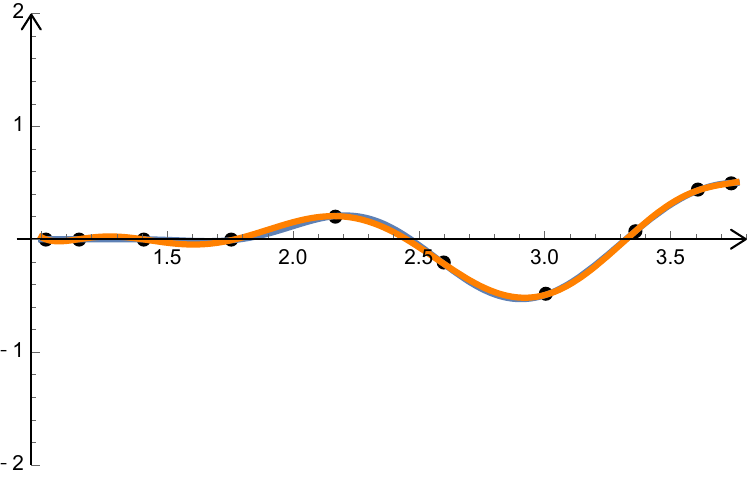} 
	&\includegraphics[keepaspectratio, scale=0.35]{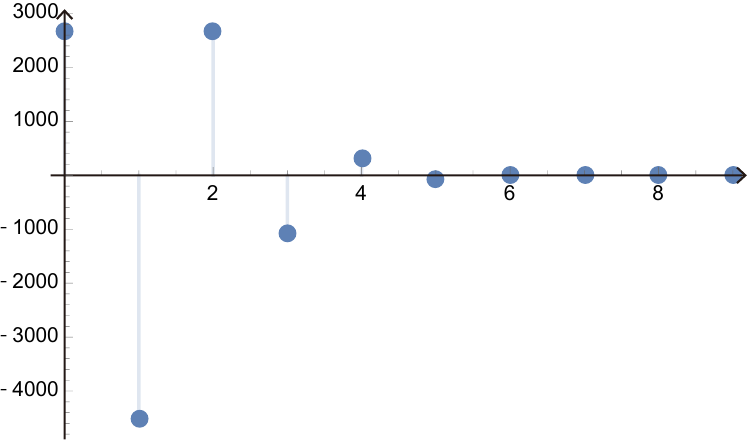} \\
	(a)&(b)&(c) \\
\end{tabular}
\end{center}
\caption{Results of a numerical simulation of the approximation for $W$ by the linear combination of $\cosh j(L-|x|)$. 
We set $W(x)=e^{-5x^2}(\cos(3\pi) x- 1/2 \cos(2\pi x))$, and $L=2$.
(a) Profiles of $W$ and the linear sum of $\cosh j(L-|x|)$. (b) Profiles of $f$ and the Lagrange interpolation polynomial on $[1, \cosh L]$. (c) Distribution of $\{ \alpha^9_j \}_{j=0}^9$.}
\label{fig:HCSE}
\end{figure}
The linear combination of $\cosh j(L-|x|)$ covers a potential $W$.
The longer the length of the interval $L$ becomes,  the worse the rate of convergence becomes from the numerical simulations.
However, as the rate of convergence is exponential as given by Theorem \ref{thm:HCSE},   the method for determining the coefficient of $a_j$ is compatible with, and useful for numerical simulations.

Figure \ref{fig2} shows the numerical results of \eqref{nonlocal:FP} with the potential $W(x)=e^{-5x^2}$ and \eqref{eq:system} with  parameters $\eps=0.001$ and $\{\alpha_j^6\}^6_{j=1}$ specified by Theorem \ref{thm:HCSE}.
We can observe that the solution $\rho$ in \eqref{nonlocal:FP} is approximated by that $\rho^\eps$ in \eqref{eq:system}, even though there are seven auxiliary factors $v_j^\eps$.
In Figure \ref{fig2} (c) shows the profiles of both $W(x)=e^{-5x^2}$ and $\sum_{j=0}^6 \alpha_j^6 \cosh(j(L-|x|))$.
Since $\sum_{j=0}^6 \alpha_j^6 \cosh(j(L-|x|))$ has good accuracy for the approximation of $e^{-5x^2}$, both curves are seen to overlap.
\begin{figure}[bt]
\begin{center}
\begin{tabular}{cccc}
	\includegraphics[keepaspectratio, scale=0.12]{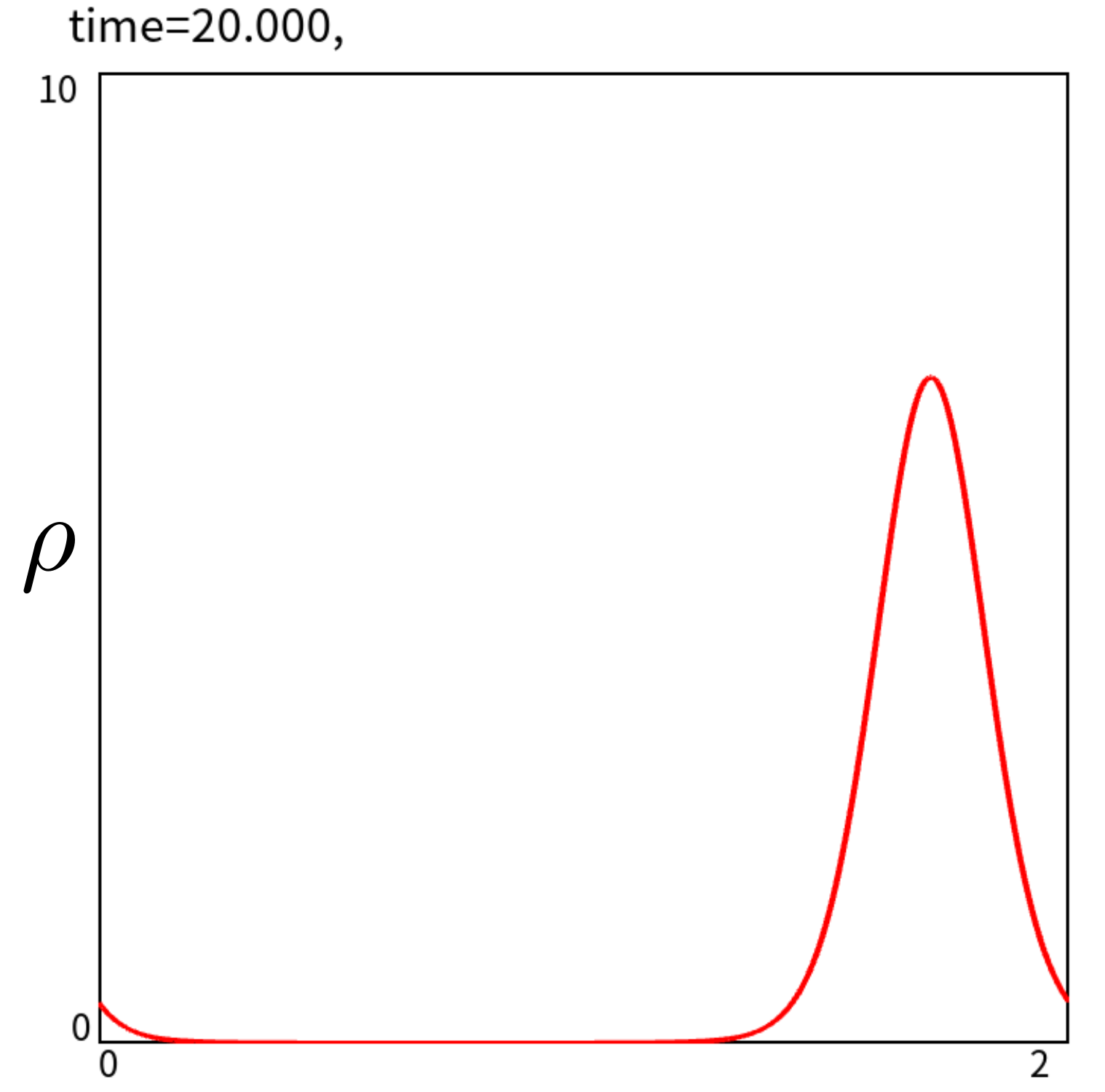}
	&\includegraphics[keepaspectratio, scale=0.12]{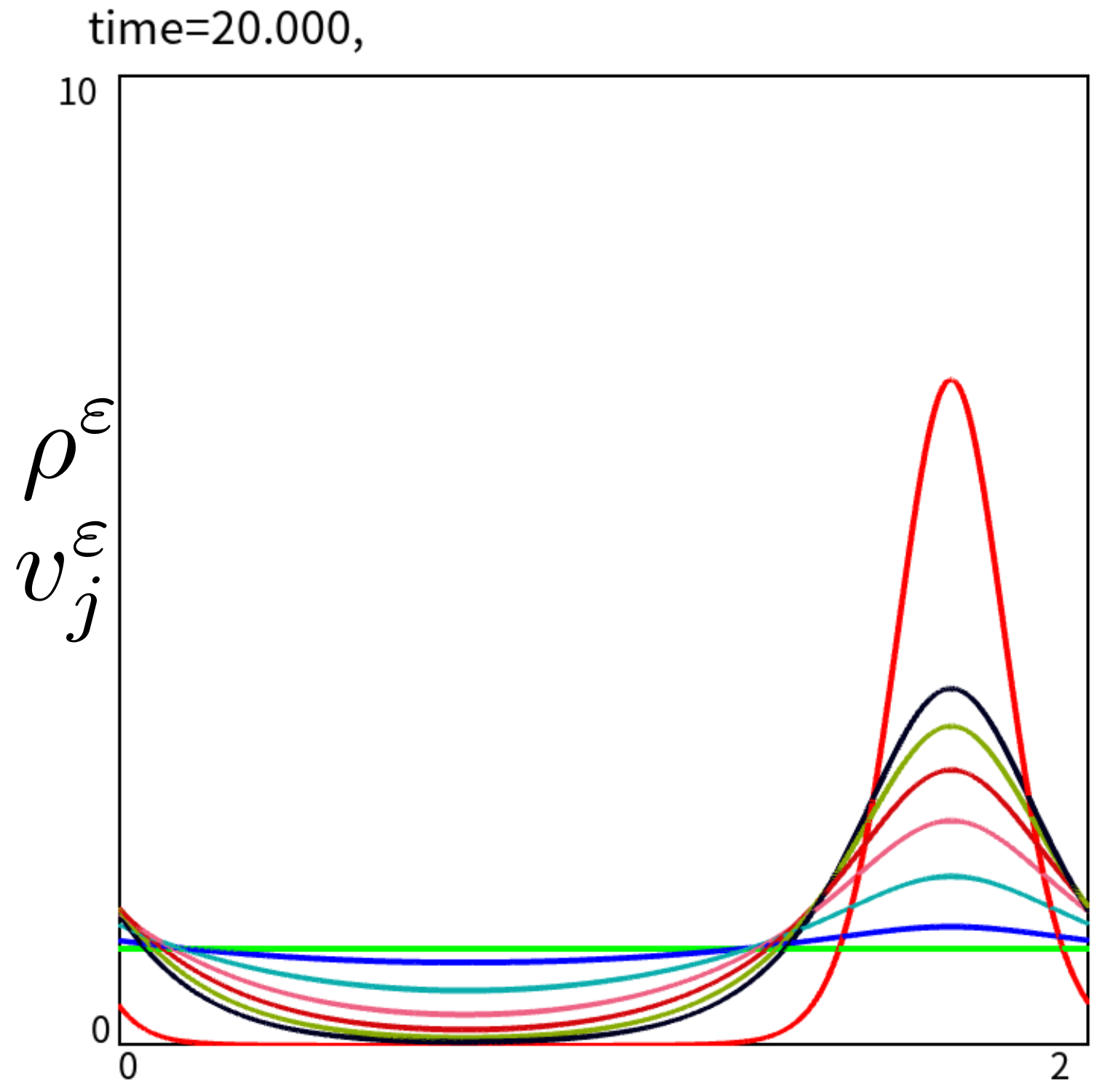} 
	&\includegraphics[keepaspectratio, scale=0.25]{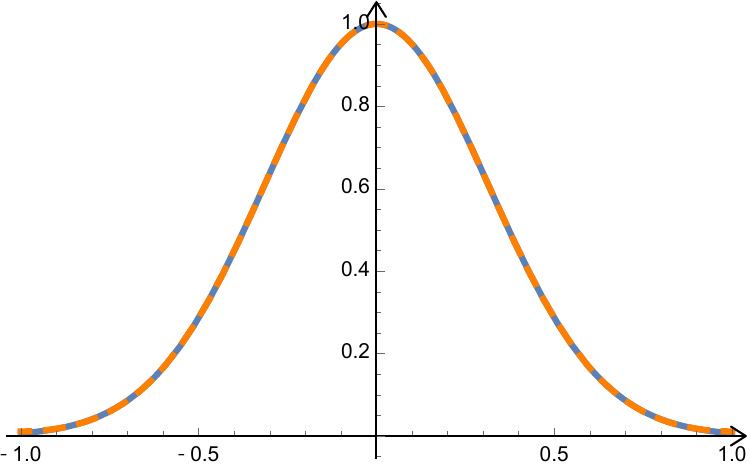} 
	&\includegraphics[keepaspectratio, scale=0.3]{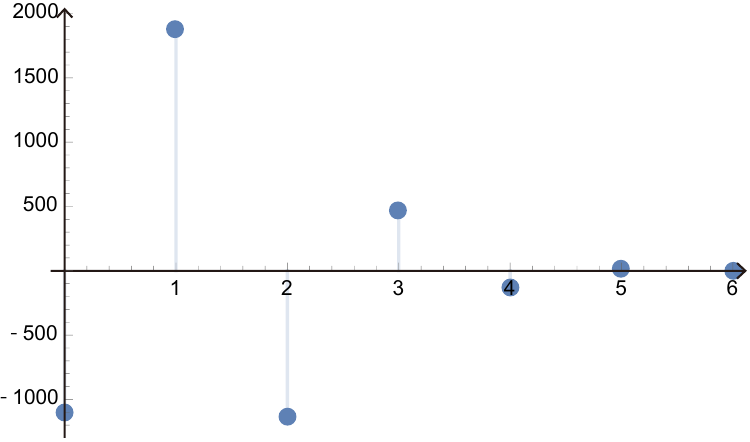} \\
	(a)&(b)&(c)&(d) \\
\end{tabular}
\end{center}
\caption{\small{ Results of numerical simulations for \eqref{nonlocal:FP} with a potential $W(x) = e^{-5x^2}$ and $\mu$ defined in \eqref{NFP:mu}, and \eqref{eq:system} with $M=7$. 
The parameters are given by $L=1$, $\eps=0.001$, $d_1=1000000$, $\mu=5$ and $d_j$ and $a_j$ are provided by \eqref{set:dj} and \eqref{set:aj}, respectively.
(a) Profiles of the numerical result of \eqref{nonlocal:FP} at $t=20.0$.
The horizontal and vertical axes correspond to the position $x$ and $\rho$, respectively.
The red curve is the numerical result of $\rho$.
(b) Profiles of the numerical result of \eqref{eq:system} at $t=20.0$.
We impose the same initial data for $\rho^\eps$ as that of $\rho$ and $(v_j)_0= k_j*\rho_0, \ (j=1,\ldots,M)$.
The axes are set same as that of (a).
The red and the other color curves correspond to $(\rho^\eps,\{v_j^\eps\}_{j=1}^7)$, respectively.
(c) Profiles of $W$ and $\sum_{j=0}^6 \alpha_j^6 \cosh(j(L-|x|))$.
The orange dashed and blue curves corresponding to $W$ and $\sum_{j=0}^6 \alpha_j^6 \cosh(j(L-|x|))$, respectively are drawn in a same plane.
(d) The distribution of $\{\alpha_j^6\}_{j=0}^6$.
}}
\label{fig2}
\end{figure}

\section{Linear stability analysis}\label{sec:LSA}
In this section, we perform a linear stability analysis around the equilibrium point for \eqref{nonlocal:FP} and \eqref{eq:system} with two or three components.
We demonstrate that the role of the advective nonlocal interactions with $W$ in the pattern formation.
We also demonstrate that the eigenvalue of the linearized operator of \eqref{eq:system} converges to that of \eqref{nonlocal:FP} as $\eps \to 0$ when the integral kernel is given by $k_j$ of \eqref{fund_sol:v_j}.
We analyze the following equation with the parameter $\mu>0$:
\begin{equation}\label{NFP:mu}\tag{$\mbox{P}_\mu$}
	\frac{ \partial \rho }{ \partial t }  = \frac{ \partial^2  \rho }{ \partial x^2 }  - \mu \frac{ \partial  }{ \partial x } \Big(  \rho  \frac{ \partial  }{ \partial x } ( W*\rho ) \Big) \ \text{in} \ \Omega\times (0,\infty).
\end{equation}
We explain the instability of the solution near the equilibrium point.
Let $\rho^*>0$ and $ \xi=\xi(x,t)$ be an arbitrary constant and a small perturbation, respectively.
 $\rho = \rho^*$ becomes a constant stationary solution of \eqref{nonlocal:FP}.
Putting $\rho(x,t) = \rho^* + \xi(x,t)$ and substituting it for \eqref{NFP:mu}, we have
\begin{align*}
	\xi_t 
	& = \xi_{xx} -  \mu \frac{ \partial }{\partial x} \Big( (\xi + \rho^*)  \frac{ \partial }{\partial x} W * (\xi + \rho^*)  \Big) \\
	& =  \xi_{xx} -  \mu( \rho^* W*\xi_{xx} + \xi_xW*\xi_x + \xi W*\xi_{xx}).
\end{align*}
Focusing on the linear part of above, we denote the linear operator $\mathcal L$ by
\begin{equation*}
	\mathcal{L}[u]:=  u_{xx} -  \mu\rho^* W*u_{xx}.
\end{equation*}
Because this linearized operator has $\mu\rho^*$, the effects of the strength of aggregation and the mass volume on the pattern formation around the constant stationary solution are equivalent.
Therefore, we replace $ \mu \rho^*$ with $\mu$.
Defining the Fourier coefficient of $W$ as 
\begin{equation*}
	\omega_n : = \frac{1}{ \sqrt{2L} }\int_\Omega W(x) e^{-i\sigma_n x} dx, \quad n \in \Z,
\end{equation*}
we have the following lemma with respect to the eigenvalues and eigenfunctions:
\begin{lemma}
Setting the eigenvalues 
\begin{equation*}
	\lambda( n ) = -\sigma_n^2( 1 - \sqrt{2L}  \mu \omega_n),
\end{equation*}
then we have
\begin{equation*}
	\mathcal L[ e^{i \sigma_n x } ] = \lambda( n ) e^{i \sigma_n x }, \quad n \in \Z.
\end{equation*}
\end{lemma}
\begin{proof}
The proof follows from a direct calculation
\begin{align*}
	&(e^{ i \sigma_n x })_{xx} = -\sigma_n^2 e^{i \sigma_n x }, \\
	&W*(e^{ i \sigma_n \cdot })_{xx} = -\sigma_n^2 \int_\Omega W(y) e^{i \sigma_n (x-y) } dy= - \sqrt{2L} \sigma_n^2 \omega_n e^{i \sigma_n x }.
\end{align*}
\end{proof}
Using this lemma, we find the solution to $\xi_t = \mathcal L [\xi]$ around $\rho^*$ in the form of $ \sum_{n\in \Z} \hat{\xi}_n e^{\lambda_n t} e^{i \sigma_n x } $, where $\{ \hat{\xi}_n \}$ is the Fourier coefficient.

Here, we recall the concept of the diffusion-driven instability in pattern formations proposed by Turing \cite{T1952}.
Diffusion-driven instability is a paradox where diffusion, typically leading to concentration homogenization, destabilizes the uniform stationary solution and induces nonuniformity due to the difference in the diffusion coefficients.
By using the eigenvalue $\lambda=\lambda(n)$ for the linear operator of the reaction-diffusion system, the diffusion-driven instability can be defined as the eigenvalue  $\lambda$ satisfies $\lambda(0)<0$ and there exists $n\in \Z$ such that $\lambda(n)>0$.
For the model \eqref{NFP:mu}, we have the following proposition:
\begin{proposition}
Suppose that $W$ satisfies that $\lim_{n \to \pm \infty} \omega_n=0$, and that there exists $0 \neq n_1 \in \N$ such that $\ds  Re \omega_{n_1} > 0$.
Then there exists $\mu_*>0$ such that for any $\mu>\mu_*$ there exists $0 \neq n_2 \in \N$ such that $\lambda(n_2) > 0$.
Therefore, the equilibrium point becomes unstable.
\end{proposition}
\begin{proof}
We see that $\lambda(0)=0$.
Next, from $Re \omega_{n_1}>0$, solving the inequality $	-\sigma_n^2(1 - \mu \omega_{n_1}) <0$ with respect to $\mu$,
we have
\begin{equation*}
	\mu > \frac{1}{ \omega_{n_1}} =: \mu_*.
\end{equation*}
Here we defined $\mu_*$.
Using this $\mu_*$, we find that for any $\mu>\mu_*$ there exists $n_2 \in \N$ such that $\lambda(n_2)>0$.
\end{proof}

We performed numerical simulations of  \eqref{NFP:mu} with the integral kernel \eqref{kernel:lali} and \eqref{pot:attract} with the finite volume method.
Figures \ref{fig1} and \ref{fig3} present the results. 
The Fourier coefficients for the integral kernels \eqref{kernel:lali} and \eqref{pot:attract} are given by 
\begin{align*}
	&\omega_1(n) = \frac{1}{\sqrt{2 L}}  \left(\frac{ 1 }{d_1 \sigma_n^2+1} - \frac{ 1 }{d_2 \sigma_n^2 +1}\right), \\
	&\omega_2(n) = \frac{\sqrt{2} }{ \sqrt{ L}  } \frac{ 1}{ \sigma_n^2  } (1-\cos (R \sigma_n ) ),
\end{align*}
respectively.
Figures \ref{fig:Mhat} (b) and  \ref{fig4} (b) show the distributions of the eigenvalues with $\omega_1(n) $ and $\omega_2(n) $.
The number of peaks of the solution at the beginning of the pattern formation corresponds to the maximum wave number in Figures \ref{fig1} and \ref{fig3}, respectively.
\begin{figure}[bt]
	\begin{center}
		\includegraphics[width=15cm, bb=0 0 754 204]{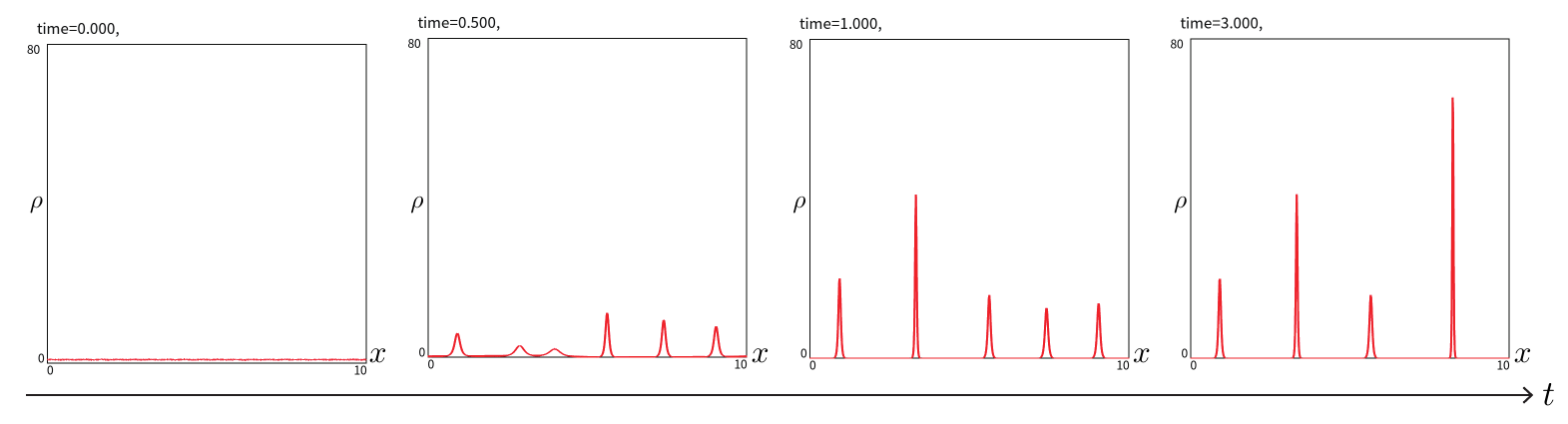}
	\end{center}
	\caption{\small{Results of a numerical simulation for \eqref{NFP:mu} with \eqref{kernel:lali}. 
	The parameters are $\mu=5.0$,  $d_1=0.1$, and $d_2=3.0$ and the initial datum are given by $1.0$ with small perturbations.
	The horizontal and vertical axes correspond to the position $x$ and value of solution $\rho$, respectively.
	The red curve corresponds to the solution $\rho$.
	The left, middle left, middle right and right pictures exhibit the profiles of solutions of \eqref{NFP:mu} with \eqref{kernel:lali} in the interval $[0, 10]$ at $t = 0, 0.5, 1.0$ and $3.0$, respectively.
	}}
	\label{fig1}
\end{figure}
\begin{figure}[bt]
\begin{center}
\begin{tabular}{cc}
	\includegraphics[width=6cm, bb=0 0 360 229]{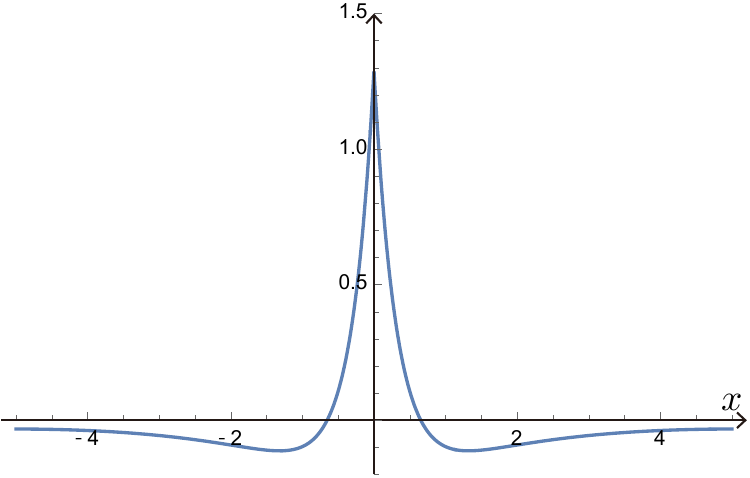}
	&\includegraphics[width=6cm, bb=0 0 559 358]{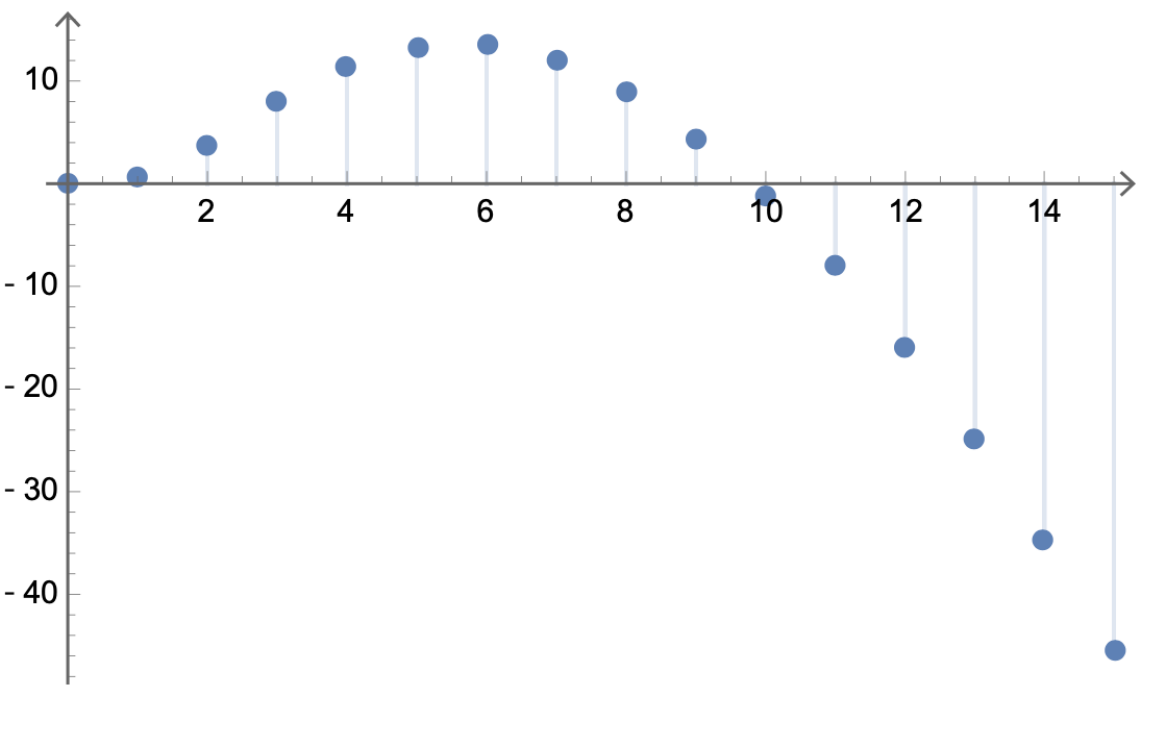} \\
	(a)&(b) \\
\end{tabular}
\end{center}
\caption{\small{ Profile of the integral kernel \eqref{kernel:lali}, and distribution of eigenvalue with $\omega_1(n)$ with same parameters as those in Fig \ref{fig1}.
}}
\label{fig:Mhat}
\end{figure}
\begin{figure}[bt]
	\begin{center}
		\includegraphics[width=15cm, bb=0 0 753 193]{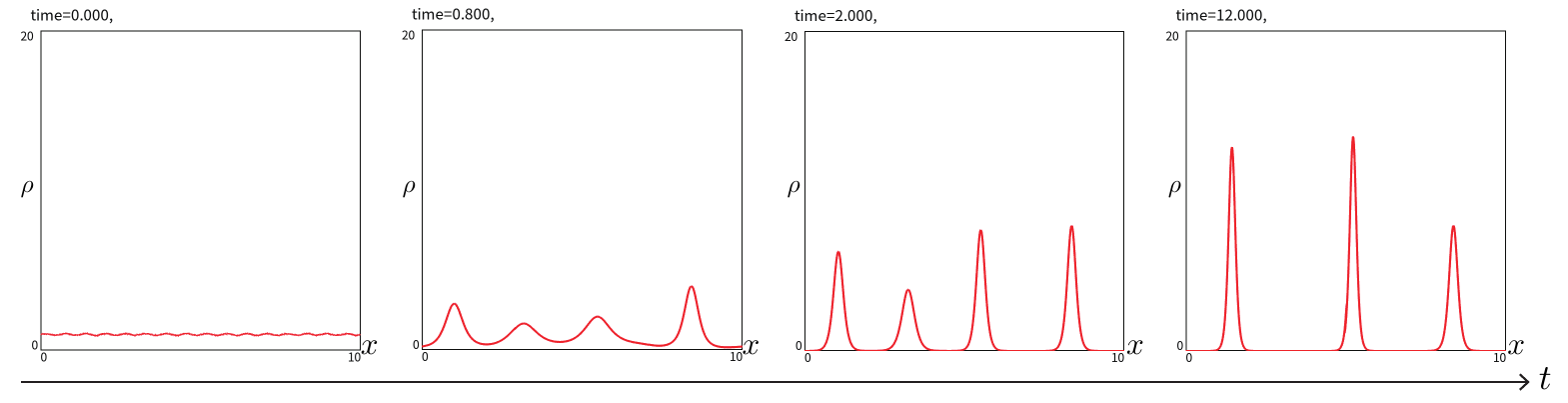}
	\end{center}
	\caption{\small{Results of a numerical simulation for \eqref{NFP:mu} with \eqref{pot:attract}. 
	The parameters are $\mu=4.0$ and $R=1.0$ and the initial datum are given by $1.0$ with small perturbations.
	The horizontal and vertical axes correspond to the position $x$ and value of solution $\rho$, respectively.
	The red curve corresponds to the solution $\rho$.
	The left, middle left, middle right and right pictures exhibit the profiles of solutions of \eqref{NFP:mu} with \eqref{pot:attract} in the interval $[0, 10]$ at $t = 0, 0.8, 2.0$ and $12.0$, respectively.
	}}
	\label{fig3}
\end{figure}
\begin{figure}[bt]
\begin{center}
\begin{tabular}{cc}
	\includegraphics[width=6cm, bb=0 0 360 245]{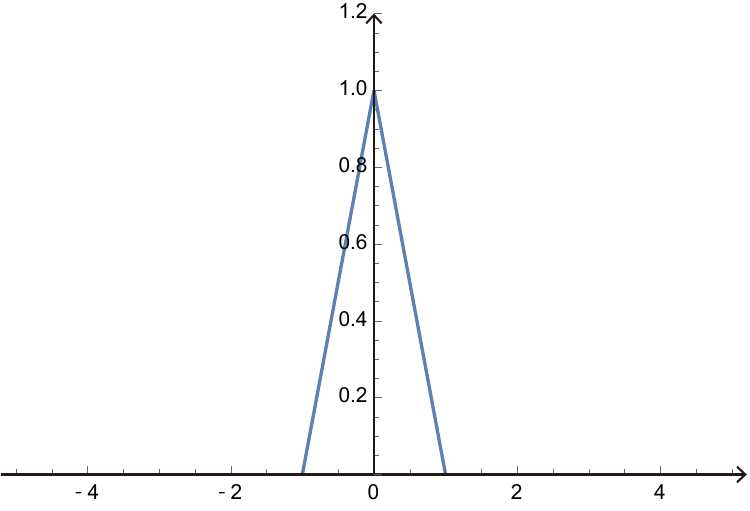}
	&\includegraphics[width=6cm, bb=0 0 287 184]{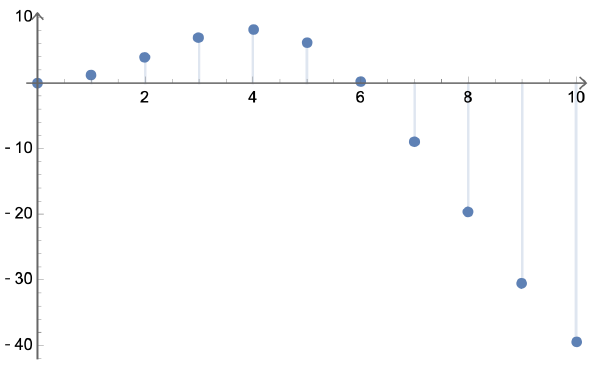} \\
	(a)&(b) \\
\end{tabular}
\end{center}
\caption{\small{ Profile of \eqref{pot:attract} and distribution of eigenvalue with $\omega_2(n)$ with same parameters as those in Fig \ref{fig3}.
}}
\label{fig4}
\end{figure}

For \eqref{kernel:KS}, we see that the solution of \eqref{NFP:mu} can be approximated by the following classical Keller--Segel equation with a linear sensitive function from Theorem \ref{thm:order_est}
\begin{equation} \label{eq:2c-system} 
\left\{
\begin{aligned}
	\frac{ \partial \rho^{\eps}}{ \partial t } & =  \frac{ \partial^2  \rho^{\eps} }{ \partial x^2 }  -  \mu\frac{ \partial  }{ \partial x } \Big(  \rho^{\eps}  \frac{ \partial   v_1^\eps }{ \partial x } \Big),\\
	\frac{ \partial v_1^\eps }{ \partial t } &= \frac{1}{\eps} \Big( d_1 \frac{\partial^2 v_1^\eps }{ \partial x^2 }  - v_1^\eps + \rho^{\eps} \Big)\\
\end{aligned}
\right. \ \text{in} \ \Omega\times (0,\infty),
\end{equation}
where $0 < \eps \ll 1$.
The auxiliary factor $v_1^\eps$ acts as an attractive substance during the chemotactic process.
It is expected that $\rho^\eps$ aggregates the place where the gradient of $v_1^\eps$ is high and the spatially localized patterns form.
Finding the solution in the form of the Fourier series expansion, 
the linearized problem around $(\rho^\eps, v_1^\eps)=(\rho^*, \rho^*)$ can be expressed by the following system:
\begin{equation*}
	\bm{\varphi}_t
	=
	\begin{pmatrix}
	-\sigma_n^2 &  \mu \sigma_n^2 \\[2mm]
	\dfrac{1}{\eps} & \dfrac{-d_1 \sigma_n^2 -1 }{\eps}
	\end{pmatrix}
	\bm{\varphi},
	\quad 
	\bm{\varphi} := 
	\begin{pmatrix}
	(\hat{\varphi_1})_n\\[2mm]
	(\hat{\varphi_2})_n
	\end{pmatrix},
\end{equation*}
where $(\hat{\varphi_1})_n$ and $ (\hat{\varphi_2})_n$ are the Fourier coefficients of the perturbations of $\rho^\eps$ and $ v_1^\eps$, respectively.
The characteristic polynomial of this Jacobi matrix $P_1(\lambda,\eps)$ is given by
\begin{align*}
	P_1(\lambda,\eps)
	&=\lambda^2+\Big( \frac{ d_1\sigma_n^2+1 }{\eps} +\sigma_n^2 \Big)\lambda+\frac{d_1\sigma_n^4 +\sigma_n^2-\mu \sigma_n^2}{\eps}\\
	&=:\lambda^2+\Big( \frac{ C_{30} }{\eps} + C_{31} \Big)\lambda + \frac{C_{32}}{\eps},
\end{align*}
where $C_{30}:=d_1\sigma_n^2+1$, $C_{31}:=\sigma_n^2 $, $C_{32}:= d_1\sigma_n^4+\sigma_n^2-\mu \sigma_n^2$.
Since
\begin{equation*}
	\eps \frac{\partial }{\partial \lambda}P_1(\lambda,\eps)
	=  2\eps \lambda +( C_{30}+\eps C_{31} )\to C_{30} = d_1\sigma_n^2+1 >0 \quad(\eps \to 0+0),
\end{equation*}
only one eigenvalue converges to a bounded value when $\eps \to 0+0$ from the implicit function theorem.
The eigenvalues are denoted by $\alpha_\eps = \alpha_\eps(n)$ and another one by $\beta_\eps = \beta_\eps (n)$.

$\alpha_\eps$ and $\beta_\eps$ are calculated as
\begin{align*}
	\alpha_\eps 
	&= \frac{1}{2} \Big( -\Big( \frac{ C_{30} }{ \eps } +C_{31} \Big) + \sqrt{ \Big( \frac{ C_{30} }{ \eps } +C_{31} \Big)^2 - \frac{ 4C_{32} }{ \eps }  } \Big)\\
	&= \frac{1}{2} \frac{ 4C_{32} }{ -( C_{30} +\eps C_{31} ) - \sqrt{ ( C_{30} +\eps C_{31} )^2 -4\eps C_{32} } },\\
	\beta_\eps
	&= \frac{1}{2} \Big( -\Big( \frac{ C_{30} }{ \eps }  + C_{31} \Big) - \sqrt{ \Big( \frac{ C_{30} }{ \eps } +C_{31} \Big)^2 - \frac{ 4C_{32} }{ \eps }  } \Big)\\
	&= \frac{1}{2} \frac{ 4C_{32} }{ -( C_{30} +\eps C_{31} ) + \sqrt{ ( C_{30} +\eps C_{31} )^2 -4\eps C_{32} } }.
\end{align*}
Thus, we find that $\lim_{\eps \to 0+0} \alpha_\eps = -\sigma_n^2 ( 1 -\sqrt{2L}\mu(\hat {k}_j)_n) =: \alpha_0(n)$ and $\lim_{\eps \to 0+0} \beta_\eps= -\infty$.
This implies that the pattern source provided by the perturbed equilibrium point is formed along the eigenfunction associated with the maximum eigenvalue of $\alpha_0(n)$.
Furthermore, the behavior of the solution $\rho$ in \eqref{NFP:mu} is extremely close to that of the Keller--Segel system of \eqref{eq:2c-system}.

For \eqref{kernel:lali}, by introducing $v^\eps_1=k_1*\rho$ and $v^\eps_2=k_2*\rho$ into $W$, the solution of \eqref{NFP:mu} can be approximated by that of the 3-component Keller--Segel system from Theorem \ref{thm:order_est}:
\begin{equation} \label{eq:3c-system} 
\left\{
\begin{aligned}
	\frac{ \partial \rho^{\eps}}{ \partial t } & =  \frac{ \partial^2  \rho^{\eps} }{ \partial x^2 }  - \mu \frac{ \partial  }{ \partial x } \Big( \rho^{\eps}  \frac{ \partial  }{ \partial x } (  v_1^\eps -  v_2^\eps ) \Big),\\
	\frac{ \partial v_1^\eps }{ \partial t } &= \frac{1}{\eps} \Big( d_1 \frac{\partial^2 v_1^\eps }{ \partial x^2 }  - v_1^\eps + \rho^{\eps} \Big),\\
	\frac{ \partial v_2^\eps }{ \partial t } &= \frac{1}{\eps} \Big( d_2 \frac{\partial^2 v_{1}^\eps }{ \partial x^2 }  - v_2^\eps + \rho^{\eps} \Big)
\end{aligned}
\right.\ \text{in} \ \Omega\times (0,\infty)
\end{equation}
with $0 < \eps \ll 1$.
In \eqref{eq:3c-system}, $v_1^\eps$ and $v_2^\eps$ represent the attractive and repulsive substances in chemotactic process, respectively.
By determining  the solution as a Fourier series expansion, the linearized problem is provided by the following system:
\begin{equation*}
	\bm{\varphi}_t
	=
	\begin{pmatrix}
	-\sigma_n^2 & \mu\sigma_n^2 & - \mu\sigma_n^2  \\[2mm]
	\dfrac{1}{\eps} & \dfrac{-d_1 \sigma_n^2 -1 }{\eps} & 0\\[2mm]
	\dfrac{1}{\eps} & 0 & \dfrac{-d_2 \sigma_n^2 -1 }{\eps} 
	\end{pmatrix}
	\bm{\varphi},
	\quad 
	\bm{\varphi} := 
	\begin{pmatrix}
	(\hat{\varphi_1})_n\\[2mm]
	(\hat{\varphi_2})_n\\[2mm]
	(\hat{\varphi_3})_n
	\end{pmatrix},
\end{equation*}
where $(\hat{\varphi_3})_n$ is the Fourier coefficient for  the perturbations of $v_2^\eps$.
The characteristic polynomial is given by
\begin{align*}
	&P_2(\lambda,\eps)= -\lambda^3 - \Big( C_{31} + \frac{ C_{33} }{ \eps } \Big) \lambda^2 - \Big( \frac{ C_{34} }{\eps^2} +\frac{ C_{35} }{\eps}  \Big)\lambda - \frac{ C_{36} }{\eps^2}\\
	&C_{33}:=2+d_1\sigma_n^2+d_2\sigma_n^2,\\
	&C_{34}:= (1+d_1\sigma_n^2)(1+d_2\sigma_n^2 ),\\
	&C_{35}:=\sigma_n^2( 2+ d_1\sigma_n^2 + d_2\sigma_n^2 ),\\
	&C_{36}:=\sigma_n^2(1+d_1\sigma_n^2)(1+d_2\sigma_n^2 ) + \mu( d_1 -d_2 )\sigma_n^4.
\end{align*}
Then, we can see that
\begin{align*}
	\eps^2\frac{\partial P_2}{\partial \lambda} 
	&= -3\eps^2\lambda^2 -2\eps \Big( \eps C_{31} + C_{33} \Big) \lambda -(  C_{34} +\eps C_{35}) \\
	&\to-  (1+d_1\sigma_n^2)(1+d_2\sigma_n^2 ) <0 \quad (\eps \to 0+0).
\end{align*}
Similarly to the previous calculation, only one eigenvalue converges to a bounded value when $\eps \to 0+0$ from the implicit function theorem.
Setting the eigenvalue as $\lambda_\eps$, we can compute that 
\begin{equation*}
	\lim_{\eps \to 0+0}\eps^2 P_2(\lambda_\eps,\eps) = - C_{34} \lambda_\eps - C_{36} =0,
\end{equation*}
and thus
\begin{equation*}
	\lambda_{0} = -\sigma_n^2 + \frac{ \mu ( d_2 -d_1 )\sigma_n^4 }{ (1+d_1\sigma_n^2)(1+d_2\sigma_n^2 ) } = -\sigma_n^2( 1- \mu \sqrt{2L} \omega_n ).
\end{equation*}
This implies that the solution to \eqref{eq:3c-system} and \eqref{NFP:mu} with \eqref{kernel:lali} are sufficiently close, but also that the Fourier mode of \eqref{NFP:mu} when the pattern forms around an equilibrium point is also extremely close to that of the 3-component attraction-repulsion Keller--Segel system \eqref{eq:3c-system}.
Because the constant stationary solution is destabilized by the auxiliary factors $v_1$ and $v_2$, the mechanism of the pattern formation is almost the same as that of diffusion-driven instability. 
In other words, if the integral kernel is provided by \eqref{kernel:lali}, the solution to \eqref{nonlocal:FP} is sufficiently close to that of the Keller--Segel system, which can cause the diffusion-driven instability, and thereby suggesting that the kernel $W$ is crucial in generating the diffusion-driven instability in the nonlocal Fokker--Planck equation \eqref{nonlocal:FP}.

We performed a numerical simulation of \eqref{eq:3c-system} with $\eps=0.001$.
The profile of the solution $\rho^\eps$ at each time point is similar to that of $\rho$ in Figure \ref{fig1} and \ref{fig5}.
As explained above, by approximating the dynamics of nonlocal evolution equations using Keller--Segel systems, we can describe the nonlocal dynamics within the framework of local dynamics, and identify both mechanisms.
\begin{figure}[bt]
	\begin{center}
		\includegraphics[width=15cm, bb=0 0 759 198]{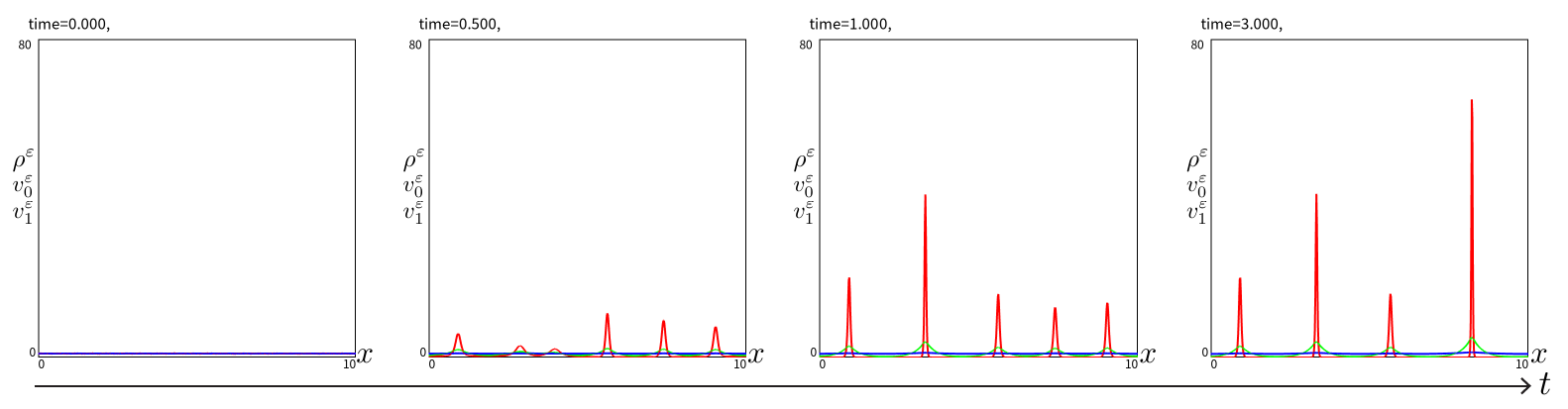}
	\end{center}
	\caption{\small{The results of a numerical simulation for \eqref{eq:system}. 
	The parameters $\mu, d_1, d_2$ and the initial data are same as that in Figure \ref{fig1}, and $\eps=0.001$ and $((v_1)_0,(v_2)_0)=( k_1*\rho_0, k_2*\rho_0 )$.
	The horizontal and vertical axes correspond to the position $x$ and value of solutions $\rho^\eps$, $v_1^\eps$ and $v_2^\eps$, respectively.
	The red, green and blue  curves correspond to the solution $\rho^\eps$, $v_1^\eps$ and $v_2^\eps$, respectively.
 The left, middle left, middle right and right pictures exhibit the profiles of solutions of \eqref{eq:system} in the interval $[0, 10]$ at $t = 0, 0.5, 1.0$ and $3.0$, respectively.
	}}
	\label{fig5}
\end{figure}

\section{Concluding remarks}\label{sec:7}
We approximated the solutions of the nonlocal Fokker--Planck equation with any even advective nonlocal interactions \eqref{nonlocal:FP} by those of multiple components of the Keller--Segel system \eqref{eq:system}. 
This indicates that the mechanism of the weight function for determining the velocity by sensing the density globally in space can be realized by combining multiple chemotactic factors.
Additionally, our results show that this diffusion–aggregation process can be described as a chemotactic process.
We propose a method in which the parameters $\{d_j, a_j\}$ can be determined based on the profile of the potential $W$.
Using the Keller--Segel type approximation, we rigorously demonstrate that the destabilization of the solution near equilibrium points in the nonlocal Fokker--Planck equation closely resembles diffusion-driven instability.
This type of analysis can be applied to other nonlocal evolution equations with advective nonlocal interactions, such as cell adhesion models.

The Keller--Segel approximation also benefits the numerical algorithm in \eqref{nonlocal:FP}.
By approximating the potential $W$ by $\sum_{j=1}^Ma_j k_j$ using Theorem \ref{thm:HCSE} and solving \eqref{eq:system} numerically, we can remove the nonlocality from \eqref{nonlocal:FP}.
By calculating these local systems instead of \eqref{nonlocal:FP} using a simple integral scheme, a numerical simulation can be performed more rapidly.

Theorem \ref{thm:RDSA} indicates that local dynamics, such as the Keller--Segel system, and nonlocal dynamics, such as the nonlocal Fokker--Planck equation, can be bridged.
Thus, we can treat the problem \eqref{nonlocal:FP} within the framework of \eqref{eq:system} if \eqref{eq:system} is easier.
As demonstrated by the linear stability analysis of \eqref{eq:system}, we can characterize the solutions to \eqref{nonlocal:FP} for local dynamics.

According to Ninomiya et al. \cite{NTY2017}, the existence of parameters $\{ a_j\}$ was shown in the weaker condition, that is, for the continuous integral kernel $W$.
Figure \ref{pot:Hapt} shows the numerical results of the approximation of \eqref{pot:attract} using the linear sum of $\cosh j(L-|x|)$.
This suggests that the condition of Theorem \ref{thm:HCSE} for determining $\{ a_j\}$ for potential $W$ may be relaxed.
We aim to intensify this investigation in the future.

\begin{figure}[bt]
\begin{center}
\begin{tabular}{ccc}
	\includegraphics[keepaspectratio, scale=0.33]{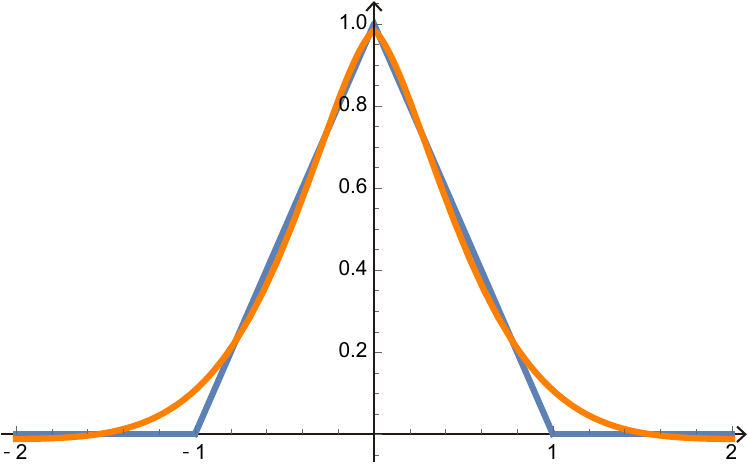}
	&\includegraphics[keepaspectratio, scale=0.33]{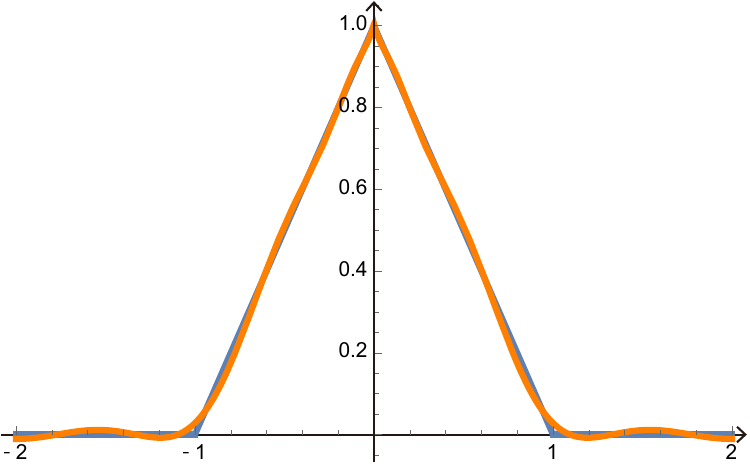} 
	&\includegraphics[keepaspectratio, scale=0.33]{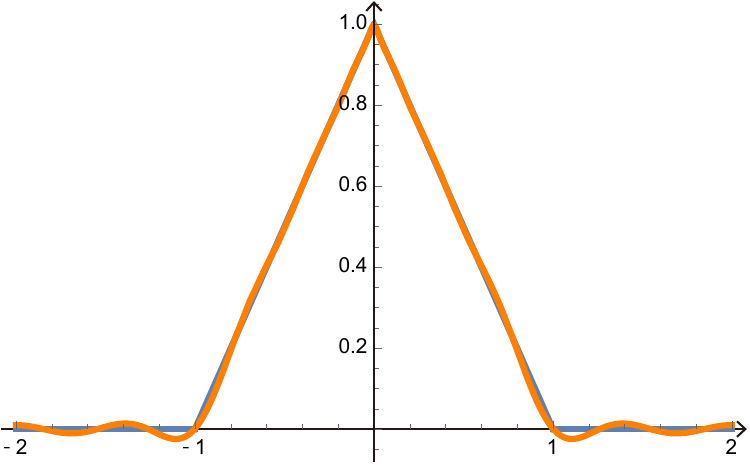} \\
	(a) $n=4$&(b) $n=8$&(c) $n=12$ \\
\end{tabular}
\end{center}
\caption{\small{Profiles of \eqref{pot:attract} and the linear combination of $\cosh j(L-|x|)$ depending of the number $n$. 
The blue and orange curves correspond to \eqref{pot:attract} and the linear combination of $\cosh j(L-|x|)$, respectively.
(a) $n=4$. (b) $n=8$. (c) $n=12$.
}}
\label{pot:Hapt}
\end{figure}

\section*{Acknowledgments}
The authors were partially supported by JSPS KAKENHI Grant Number 22K03444.
HM was partially supported by JSPS KAKENHI Grant Numbers 20H01823 and 21KK0044. 
YT was partially supported by JSPS KAKENHI Grant Number 20K14364.

\appendix
\section{Proof of Lemma \ref{lemm:bounded_G}}
\label{append:LBG}
\begin{proof}
First, we denote the Fourier Coefficients of $f$ and $g$ by
\begin{equation*}
	 f_n:=  \frac{1}{\sqrt{2L}}\int_\Omega f(x) e^{-i\sigma_n x} dx,  \quad g_n(t):=  \frac{1}{\sqrt{2L}}\int_\Omega g(x,t) e^{-i\sigma_n x} dx,
\end{equation*}
for $n \in \Z$, respectively. 

Then, using  the orthogonality and the Parseval identity, we compute that 
\begin{align*}
	\left\| G * f \right\| _{ L^2 (\Omega)} ^2 (t)
	&= \left\| \frac{1}{\sqrt{2L}} \sum_{n \in \Z} e^{- \sigma_n^2 t} e^{i \sigma_n \cdot} f_n \right\|_{  L^2 (\Omega)} ^2 \\
	&= \sum_{n \in \Z} e^{- 2 \sigma_n^2 t}  f_n^2  
	 \le  \sum_{n \in \Z}  f_n^2  
	= \left\| f \right\|_{ L^2 (\Omega)}^2.
\end{align*}

Straightforwardly  we can calculate that 
\begin{align*}
\left\|  \int_0^t  \int_\Omega  G( \cdot -y , t-s) g (y,s) dy ds  \right\| _{ L^2 (\Omega)}^2 
&= \left\| \frac{1}{\sqrt{2L}} \sum_{n \in \Z} e^{ -\sigma_n^2 t} e^{i \sigma_n \cdot } \int_0^t  e^{\sigma_n^2 s} g_n(s) ds  \right\|_{ L^2 (\Omega)} ^2\\
& = \sum_{n \in \Z} e^{ - 2 \sigma_n^2 t} \Big( \int_0^t  e^{\sigma_n^2 s} g_n(s) ds  \Big)^2 \\
& \le \sum_{n \in \Z}  \frac{1}{\sigma_n^4} ( 1- e^{-\sigma_n^2 t}  )^2 \sup_{ s \in [0, t] } | g_n(s)|^2  \\
& = t^2 \sum_{n \in \Z}  e^{ - 2 \sigma_n^2 t \theta_n} \sup_{ s \in [0, t] } |g_n(s)|^2   \\
& \le t^2  \left\|  g  \right\|_{ C([0,T], L^2 (\Omega)) }^2,
\end{align*}
where we used the Maclaurin series expansion, that is, for $\sigma_n^2 t$ there exists $\theta_n \in (0,1)$ such that 
\begin{equation}\label{series:Mac}
e^{-\sigma_n^2 t} = 1 - \sigma_n^2 t e^{-\sigma_n^2 t \theta_n}.
\end{equation}
Similarly, the Maclaurin series expansion \eqref{series:Mac} yields that 
\begin{align*}
\left\|  \int_0^t  \int_\Omega  G_x( \cdot -y , t-s) g(y,s) dy ds  \right\| _{ L^2 (\Omega)}^2 
&= \left\| \frac{1}{\sqrt{2L}} \sum_{n \in \Z} i\sigma_ne^{ -\sigma_n^2 t} e^{i \sigma_n \cdot } \int_0^t  e^{\sigma_n^2 s} g_n(s) ds  \right\|_{ L^2 (\Omega)} ^2 \\
& = \sum_{n \in \Z} \sigma_n^2e^{ - 2 \sigma_n^2 t} \Big( \int_0^t  e^{\sigma_n^2 s} g_n(s) ds  \Big)^2 \\
& \le \sum_{n \in \Z}  \frac{( 1- e^{-\sigma_n^2 t} )^2 }{\sigma_n^2}  \sup_{ s \in [0, t] } |g_n(s)|^2  \\
& =  t \sum_{n \in \Z}    e^{ -\sigma_n^2 t \theta_n } ( 1- e^{-\sigma_n^2 t}) \sup_{ s \in [0, t] } |g_n(s)|^2 \\
& \le t \sum_{n \in \Z}  \sup_{ s \in [0, T] } |g_n(s)| ^2 \\
&= t  \left\| g \right\|_{ C([0,T], L^2 (\Omega)) }^2.
\end{align*}

Next, we estimate the boundedness with $G_j^\eps$.
Utilizing  the orthogonality and the Parseval identity, we obtain that
\begin{align*}
&\frac{1}{\eps^2} \left\|  \int_0^t  \int_\Omega  G_j^\eps( \cdot -y , t-s)  g(y,s)  dy ds \right\|_{ L^2 (\Omega)}^2 \\
&= \frac{1}{\eps ^2} \left\| \frac{1}{\sqrt{2L} }   \sum_{n \in \Z}   e^{ - \frac{ d_j\sigma_n^2 + 1 }{ \eps }t  } 
	 			e^{ i\sigma_n \cdot }  \int_0^t e^{  \frac{ d_j\sigma_n^2 + 1 }{ \eps }s  }  g_n(s) ds  \right\|_{ L^2 (\Omega)} ^2 \\
	& = \frac{1}{\eps ^2}   \sum_{n \in \Z}  \Big( e^{ - \frac{ d_j\sigma_n^2 + 1 }{ \eps }t  }  \int_0^t e^{  \frac{ d_j\sigma_n^2 + 1 }{ \eps }s  }  g_n(s) ds \Big)^2 \\
	& \le  \sum_{n \in \Z}  \frac{1}{(d_j\sigma_n^2 + 1)^2}\Big( 1 - e^{ - \frac{ d_j\sigma_n^2 + 1 }{ \eps }t  }  \Big)^2 \sup_{s \in [0, t]} | g_n(s) |^2\\
	& \le  \sum_{n \in \Z} \sup_{s \in [0, T]} | g_n(s) |^2\\
	& = \left\| g \right\|_{ C([0,T], L^2 (\Omega)) }^2,
\end{align*}
and
\begin{align*}
	&\frac{1}{\eps^2} \left\|  \int_0^t  \int_\Omega  (G_{j}^\eps)_x( \cdot -y , t-s)  g(y,s)  dy ds \right\|_{ L^2 (\Omega)} ^2 \\
	& = \frac{1}{\eps ^2} \left\| \frac{1}{\sqrt{2L} }   \sum_{n \in \Z}  i\sigma_n e^{ - \frac{ d_j\sigma_n^2 + 1 }{ \eps }t  } 
	 			e^{ i\sigma_n \cdot }  \int_0^t e^{  \frac{ d_j\sigma_n^2 + 1 }{ \eps }s  }  g_n(s) ds  \right\|_{ L^2 (\Omega)} ^2 \\
	&\le \sum_{n \in \Z}  \frac{\sigma_n^2}{(d_j\sigma_n^2 + 1)^2}\Big( 1 - e^{ - \frac{ d_j\sigma_n^2 + 1 }{ \eps }t  }  \Big)^2 \sup_{s \in [0, t]} | g_n(s) |^2 \\
	& \le \frac{1}{d_j^2 \sigma_1^2} \sum_{n \in \Z} \sup_{s \in [0, T]} | g_n(s) |^2 \\
	& = \frac{1}{d_j^2 \sigma_1^2}  \left\| g \right\|_{C([0,T], L^2 (\Omega)) }^2.
\end{align*}

\end{proof}

\section{Proof of Lemma \ref{lemm:k_j_L^1}}\label{append:kjL1}
\begin{proof}
We can compute that 
\begin{align*}
\left\|  k_j \right\|_{ L^1 (\Omega)} = 2c_k(j)\int_0^L \cosh \frac{L-x}{\sqrt{d_j}} = 1,
\end{align*}
and
\begin{align*}
\left\|  (k_j)_x \right\|_{ L^1 (\Omega)} 
&=  \frac{c_k(j) }{ \sqrt{ d_j} }  \Big(  \int_0^L \Big| \sinh \frac{ L-x }{ \sqrt{ d_j} } \Big| dx + \int_{-L}^0 \Big| \sinh \frac{ L+x }{ \sqrt{ d_j} } \Big| dx \Big)\\
&= 2   \frac{c_k(j) }{ \sqrt{ d_j} }   \int_0^L  \sinh \frac{ x }{ \sqrt{ d_j} }  dx\\
&= C_8.
\end{align*}
We see that 
\begin{align*}
\left\|  (k_j)_{xx} \right\|_{ L^1 (\Omega)} = \frac{1}{d_j}\left\|  k_j \right\|_{ L^1 (\Omega)}
= \frac{1}{d_j}.
\end{align*}
As $(k_j)_{xxx}(x)  =  (k_j)_x(x) /d_j $, we obtain the third assertion.
Finally, we can compute that 
\begin{align*}
	\left\|  (k_j)_x \right\|_{ C (\Omega)} = \frac{ c_k(j) }{\sqrt{d_j}} * \sinh \frac{L}{\sqrt{ d_j}} = \frac{1}{2 d_j}.
\end{align*}
\end{proof}

\bibliography{references}

\end{document}